\documentclass{siamart171218}

\usepackage{jeff}
\usepgfplotslibrary{fillbetween}
\usepgfplotslibrary{groupplots}
\usepackage{pgfplotstable}
\usepackage{booktabs}

\usepackage[algo2e, lined, vlined, linesnumbered, english]{algorithm2e}

\SetCommentSty{mycommfont}
\SetKwComment{Comment}{}{}
\SetKwInOut{Input}{Input}
\SetKwInOut{Output}{Output}
\SetSideCommentRight	% Push comments to far right
\SetKwComment{comment}{}{}

\SetAlCapSkip{1pc}

\DeclareMathOperator{\Hess}{Hess}
\DeclareMathOperator{\Range}{Range}

\DeclareMathOperator{\grad}{grad}

\newcommand{\TheTitle}{Data-Driven Polynomial Ridge Approximation Using Variable Projection}
\newcommand{\TheAuthors}{Jeffrey M. Hokanson and Paul G. Constantine}
\headers{Data-Driven Polynomial Ridge Approximation}{\TheAuthors}

\title{{\TheTitle}\thanks{Submitted to the editors 20 February 2017.
\funding{This work is supported by Department of Defense, 
Defense Advanced Research Project Agency's program Enabling Quantification of Uncertainty in Physical Systems. 
The second author's work is partly supported by the U.S. Department of Energy Office of Science, 
Office of Advanced Scientific Computing Research, Applied Mathematics program under Award Number DE-SC-0011077. 
}}}
\author{Jeffrey M. Hokanson\thanks{
	Department of Computer Science, 
	University of Colorado Boulder,
	1111 Engineering Dr, Boulder, CO 80309,
	\{\email{Jeffrey.Hokanson},\email{Paul.Constantine}\}\email{@colorado.edu}}
	\and Paul G. Constantine\footnotemark[2]
}
\begin{document}
\maketitle

\begin{abstract}
Inexpensive surrogates are useful for reducing the cost of science and engineering studies 
involving large-scale, complex computational models with many input parameters. 
A ridge approximation is one class of surrogate 
that models a quantity of interest
as a nonlinear function of a few linear combinations of the input parameters.
When used in parameter studies (e.g., optimization or uncertainty quantification),
ridge approximations allow the low-dimensional structure to be exploited,
reducing the effective dimension.
We introduce a new, fast algorithm for constructing a 
ridge approximation where the nonlinear function is a polynomial.
This polynomial ridge approximation is chosen to minimize least squared mismatch between 
the surrogate and the quantity of interest on a given set of inputs.
Naively, this would require optimizing both the polynomial coefficients and 
the linear combination of weights; the latter of which define a low-dimensional subspace of the input space.
However, given a fixed subspace the optimal polynomial can be found by solving a linear least-squares problem.
Hence using \emph{variable projection} the polynomial can be implicitly defined,
leaving an optimization problem over the subspace alone.
Here we develop an algorithm that finds this polynomial ridge approximation
by minimizing over the Grassmann manifold of low-dimensional subspaces
using a Gauss-Newton method.
%We provide details of this optimization algorithm
%and demonstrate its performance on several numerical examples. 
Our Gauss-Newton method has superior theoretical guarantees and faster convergence 
on our numerical examples than the alternating approach for polynomial ridge approximation earlier proposed by
Constantine, Eftekhari, Hokanson, and Ward
[{\tt https://doi.org/10.1016/j.cma.2017.07.038}]
that alternates between (i) optimizing the polynomial coefficients given the subspace and 
(ii) optimizing the subspace given the coefficients.
\end{abstract}

\begin{keywords}
active subspaces, emulator, Grassmann manifold, response surface, ridge function, variable projection
\end{keywords}
\begin{AMS}
	49M15, % Calculus of variations and optimal control; optimization  	Newton-type methods
	62J02, %  	General nonlinear regression
	90C53 % mathematical programming,   	Methods of quasi-Newton type
\end{AMS}

\section{Introduction}
Many problems in uncertainty quantification~\cite{Smith2013,Sullivan2015} and design~\cite{SW10,Wang2006} 
involve a scalar quantity of interest derived from a complex model's output that depends on the model inputs.
Here we denote the map from model inputs to the quantity of interest as $f:\set D \subseteq \R^m\to\R$. 
In many cases evaluating the quantity of interest is too expensive to permit the number of evaluations 
needed for the design or uncertainty study---e.g., estimating the distribution of $f$ given a probability distribution for $\ve x \in \set D$ 
or optimizing $f$ over design variables $\ve x \in \set D$. 
One approach to reduce the parameter study's cost is to construct an inexpensive surrogate, 
also known as a \emph{response surface}~\cite{Jones2001,Myers1995} or an \emph{emulator}~\cite{Santner2003}, 
using pairs of inputs $\lbrace \ve x_i\rbrace_{i=1}^M \subset \set D$ 
and outputs $\lbrace f(\ve x_i)\rbrace_{i=1}^M \subset \R$. 
In high dimensions (i.e., when $m$ is large), 
building this surrogate is challenging since the number of parameters in the surrogate often grows exponentially in the input dimension;
for example, a polynomial surrogate of total degree $p$ has $\order(p^m)$ parameters (polynomial coefficients) as $m\to\infty$.
In such cases, an exponential number of samples are required to yield a well-posed problem for constructing the surrogate. 
One workaround is to use surrogate models where the number of parameters grows more slowly; 
however, to justify such low-dimensional models, 
we must assume the function being approximated admits the related low-dimensional structure. 
Here we assume that the quantity of interest $f$ varies primarily along $n<m$ directions in its $m$-dimensional input space---%
that is, $f$ has an $n$-dimensional \emph{active subspace}~\cite{Con15}. 
Such structure has been demonstrated in quantities of interest arising a wide range of computational science applications: 
integrated hydrologic models~\cite{jefferson2015active,Jefferson2017}, 
a solar cell circuit model~\cite{constantine2015discovering}, 
a subsurface permeability model~\cite{Gilbert2016}, 
a lithium-ion battery model~\cite{constantine2016time}, 
a magnetohydrodynamics power generation model~\cite{glaws2016dimension}, 
a hypersonic scramjet model~\cite{constantine2015exploiting}, 
an annular combustor model~\cite{bauerheim2016}, 
models of turbomachinery~\cite{Seshadri2017}, 
satellite system models~\cite{Hu2016}, 
in-host HIV models~\cite{Loudon2017}, 
and computational models of aerospace vehicles~\cite{Lukaczyk2015,lukaczyk2014active} and automobiles~\cite{othmer2016}.

\begin{figure}
\centering
\hspace*{0.25em}
\begin{tikzpicture}
\begin{groupplot}[
	group style = {group size = 2 by 1, horizontal sep=4.5em},
	width = 6.5cm,
	height = 4.5cm,
	ylabel = $f(\ve x)$,
	ymin = 0, ymax = 2.2,
	clip mode=individual,
	title style = {yshift = -1ex},
	]
	\nextgroupplot[xlabel=$x_{1}$, title = {Coordinate shadow plot}, xmin = -1, xmax = 1]
	\addplot[black, only marks, mark size=0.5] table [x=X1, y=f] {fig_example.dat}; 

	\nextgroupplot[xlabel=$\ma U^\trans \ve x$, title = {Subspace shadow plot}, xmin = -2, xmax = 2]
	\addplot[black, only marks, mark size=0.5] table [x=UX, y=f] {fig_example.dat};
	\addplot[colorbrewerA1, very thick] table [x=UX, y=fit] {fig_example_fit.dat};
\end{groupplot}
\end{tikzpicture} \\[-6pt]
\begin{tikzpicture}
	\begin{axis}[
		width = 13cm,
		height = 3.5cm,
		xlabel = parameter $i$,
		ylabel = {$[\ma U]_i$},
		ymax = 0.3,
		ymin = -0.3,
		xmin = 0, 
		xmax = 101, 
		xtick = {1,5,10, ..., 100},
		ytick = {-0.3, -0.2, -0.1, 0, 0.1, 0.2, 0.3},
		title = {Projection weights},
		title style = {yshift = -1ex},
		]
		\addplot[black, only marks, mark size = 1] table [x=i, y=U_fit] {fig_example_U.dat};
		\addplot[black, only marks, mark size = 2, mark=o] table [x=i, y=U] {fig_example_U.dat};
	\end{axis}
\end{tikzpicture}

\caption{
	Ridge approximations reveal structure not present in coordinate perspectives.
	In this toy example, $f: \set D :=[-1,1]^{100}\to \R$
	with $f(\ve x) = |\hve u^\trans \ve x| + 0.1(\sin(1000[\ve x]_2) + 1)$
	where $\hve u$ has been sampled uniformly on the unit sphere
	and the sine term simulates deterministic noise.
	Viewed from a coordinate perspective on the upper left, no structure is present.
	However, after fitting a ridge function to $N=1000$ samples of $f$
	with polynomial degree $p=7$ and subspace dimension $n=1$ using our \cref{alg:main},
	the structure of $f$ is revealed by looking along the recovered subspace $\ma U$.
	On the bottom, we see that the coefficients of $\ma U$ in the ridge approximation, $\bullet$, 
	closely match the coefficients of $\hve u$, $\circ$.
}
\label{fig:example}
\end{figure}

For functions with this low-dimensional structure
an appropriate surrogate model is a \emph{ridge function}~\cite{pinkus2015}:
the composition of a linear map from $\R^m$ to $\R^n$ with a nonlinear function $g$ of $n$ variables:
\begin{equation}\label{eq:ridge}
	f(\ve x) \approx g(\ma U^\trans \ve x), 
		\quad \text{where} \quad 
	\ma U\in \R^{m\times n} \text{ and } \ma U^\trans \ma U = \ma I.
\end{equation}
In this paper we consider \emph{polynomial ridge approximation}~\cite{CEHW17}
where $g$ is a multivariate polynomial of total degree $p$
and construct an efficient algorithm for \emph{data-driven polynomial ridge approximation}
that chooses $g$ and $\ma U$ to minimize the 2-norm misfit on a training set $\{(\ve x_i, f(\ve x_i))\}_{i=1}^M$:
\begin{equation}\label{eq:pra_opt}
	\minimize_{\substack{g \in \mathbb{P}^p(\R^n) \\
				 \Range \ma U \in \mathbb{G}(n, \R^m)}}
		\quad
		\sum_{i=1}^M \left[ f(\ve x_i) - g(\ma U^\trans \ve x_i)\right]^2,
\end{equation}
where $\mathbb{P}^p(\R^n)$ denotes the set of polynomials on $\R^n$ of total degree $p$
and $\mathbb{G}(n,\R^m)$ denotes the Grassmann manifold of $n$ dimensional subspaces of $\R^m$.
\Cref{fig:example} provides an example of this approximation on a toy problem.
By exploiting ridge structure, fewer samples of the quantity of interest
are required to make this approximation problem overdetermined.
For example, whereas a polynomial of total degree $p$ on $\R^m$ requires ${m+p \choose p}$ samples, 
a polynomial ridge approximation of total degree $p$ on a $n$-dimensional subspace
requires only ${n + p \choose p} + mn$.

In the remainder of this paper we develop an efficient algorithm for solving the 
least squares polynomial ridge approximation problem~\cref{eq:pra_opt} by exploiting its inherent structure. 
To begin, we first review the existing literature from the applied math and statistics communities on ridge functions
in \cref{sec:background}.
Then to start building our algorithm for polynomial ridge approximation,
we show how \emph{variable projection}~\cite{GP73} can be used to 
implicitly construct the polynomial approximation given a subspace defined by the range of $\ma U$
in \cref{sec:reformulate}.
We further address the numerical issues inherent in polynomial approximation by using a Legendre basis 
and employing shifting and scaling of the projected coordinates $\ma U^\trans \ve x_i$. 
Then, with an optimization problem posed over the subspace spanned by $\ma U$ alone, 
we use techniques for optimization on the Grassmann manifold developed by Edelman, Arias, and Smith~\cite{EAS98}. 
Due to the structure of the Jacobian, 
we are able to develop an efficient Gauss-Newton algorithm as described in \cref{sec:algorithm}. 
We compare the performance of our data-driven polynomial ridge approximation algorithm 
to the alternating approach described by Constantine et al.~\cite[Alg.~2]{CEHW17}. 
Their algorithm constructs the ridge approximation by alternating between minimizing polynomial 
given a fixed subspace and then minimizing the subspace given a fixed polynomial. 
Our algorithm provides improved performance due the faster convergence of Gauss-Newton-like algorithms 
and a more careful implementation that exploits the structure of the ridge approximation problem 
during the subspace optimization step. 
The improved performance can aid model selection studies---e.g., cross-validation---that find the best parameters 
of the approximation, such as the subspace dimension $n$ and the polynomial degree $p$; 
addressing this issue is beyond the scope of the present work. 
In \cref{sec:examples} we provide examples comparing the Gauss-Newton method 
to the alternating approach on toy problems and then further demonstrate the effectiveness of our algorithm 
by constructing polynomial ridge approximations of two $f$'s from application problems: 
an 18-dimensional airfoil model~\cite[\S5.3.1]{Con15} and a 100-dimensional elliptic PDE problem~\cite{CDW14}. 
For completeness, we compare the performance of the polynomial ridge approximation 
to alternative surrogate models (Gaussian processes and sparse polynomial approximations) 
using a testing set of samples from the physics-based models. 
However, we emphasize that our goal is not necessarily to claim that ridge approximation is 
always superior to alternative models---only that ridge approximation is appropriate for functions 
that vary primarily along a handful of directions in their domain.

\section{Related ideas and literature\label{sec:background}}
Approximating multivariate functions by ridge functions 
has been studied under various names by different applied mathematics and statistics subcommunities. 
Pinkus' monograph~\cite{pinkus2015} surveys the approximation theory related to ridge functions. 
Hastie et al.~\cite[Chapter 3.5]{ESL2009} review the general idea of building regression models 
on a few derived variables that are linear combinations of the predictors and discuss options 
for choosing the linear combination weights, e.g., principal components or partial least squares. 
In what follows, we review related ideas across the literature that we are aware of; 
some exposition mirrors the review in Constantine et al.~\cite{CEHW17}. 
Although our nonlinear least squares approach may benefit from ideas embedded in these approaches,
their precise application is outside the scope of this paper.

\subsection{Projection pursuit regression}
In the context of statistical regression, 
Friedman and Stuetzle~\cite{Friedman1981} proposed \emph{projection pursuit regression} 
which models the quantity of interest as a sum of one-dimensional ridge functions:
%with a ridge function model of the link function:
\begin{equation}
	\label{eq:ppr}
	y_i \;=\; \sum_{k=1}^r g_k(\ve u_k^T\ve x_i) + \varepsilon_i, \qquad \ve u_k \in \R^m
\end{equation}
where $\ve x_i$'s are samples of the predictors, $y_i$'s are the associated responses, 
and $\varepsilon_i$'s model random noise---all standard elements of statistical regression~\cite{Weisberg2005}. 
The $g_k$'s are smooth univariate functions (e.g., splines), 
and the $\ve u_k$'s are the directions of the ridge approximation. 
To fit the projection pursuit regression model, one minimizes the mean-squared error over the directions $\{\ve u_k\}$ 
and the parameters of $\{g_k\}$. 
Motivated by the projection pursuit regression model, Diaconis and Shahshahani~\cite{Diaconis1984} 
studied the approximation properties of nonlinear functions ($g_k$ in \eqref{eq:ppr}) 
of linear combinations of the variables ($\ve u_k^\trans\ve x$ in \eqref{eq:ppr}). 
Huber~\cite{Huber1985} surveyed a wide class of projection pursuit approaches across an array of multivariate problems; 
by his terminology, \emph{ridge approximation} could be called \emph{projection pursuit approximation}. 
Chapter 11 of Hastie et al.~\cite{ESL2009} links projection pursuit regression to neural networks, 
which uses ridge functions with particular choices for the $g_k$'s (e.g., the sigmoid function). 
Although the optimization problem may be similar, 
the statistical regression context is different from the approximation context, 
since there is no inherent randomness (e.g., $\varepsilon_i$ in \cref{eq:ppr}) in the approximation problem. 

\subsection{Sufficient dimension reduction}
In the context of statistical regression
there is also a vast body of literature devises methods for finding a low-dimensional linear subspace 
of the predictor space that is statistically sufficient to characterize the predictor/response relationship; 
see the well-known text by Cook~\cite{Cook1998} and the more modern review by Adragni and Cook~\cite{Adragni2009}. 
From this literature, the minimum average variance estimate (MAVE) method~\cite{Xia2002} 
uses an optimization formulation similar to \eqref{eq:pra_opt} to identify the dimension reduction subspace. 
However, the space of functions for ridge approximation in this approach is 
local linear models---as opposed to a global polynomial of degree $p$---where locality is imposed by kernel-based 
weights in the objective function centered around each data point. 
A similar approach was used in Xia's multiple-index model for regression~\cite{Xia2008}.

\subsection{Gaussian processes with low-rank correlation models}
In Gaussian process regression~\cite{gpml2006}, 
the conditional mean of the Gaussian process model given data (e.g., $\{y_i\}$ as in \eqref{eq:ppr}) 
is the model's prediction. 
This conditional mean is a linear combination of radial basis functions with centers at a set of points $\{\ve x_i\}$, 
where the form of the basis function is related to the Gaussian process' assumed correlation. 
Vivarelli and Williams~\cite{vivarelli99} proposed a correlation model of the form
\begin{equation}
	C(\ve x,\ve x') \;\propto\;
	\exp\left[
	-\frac{1}{2}(\ve x-\ve x')^\trans\ma U\ma U^\trans(\ve x-\ve x')
	\right],
\end{equation}
where $\ma U$ is a tall matrix. 
In effect, the resulting conditional mean is a function of linear combinations of the predictors, 
$\ma U^\trans \ve x$---i.e., a ridge function. 
A maximum likelihood estimate of $\ma U$ is the minimizer of an optimization problem similar to \eqref{eq:pra_opt}. 
Bilionis et al.~\cite{bilionis2016gaussian}, use a related approach from a Bayesian perspective 
in the context of uncertainty quantification, where the subspace defined by $\ma U$ enables powerful dimension reduction. 
And Liu and Guillas~\cite{Liu2017} develop a related low-dimensional model with linear combinations of predictors 
for Gaussian processes; 
their approach leverages gradient-based dimension reduction proposed by Fukumizu and Leng~\cite{Fukumizu2014} 
to find the linear combination weights.

\subsection{Ridge function recovery}
Recent work in constructive approximation seeks to recover the parameters of a ridge function 
from point queries~\cite{Fornasier2012,cohen2012,Tyagi2014};
that is, determine $\ma U$ in the ridge function $f(\ve x)=g(\ma U^\trans \ve x)$ using pairs $\{\ve x_i,f(\ve x_i)\}$.
Algorithms for identifying $\ma U$ (e.g., Algorithm 2 in~\cite{Fornasier2012}) are quite different 
than optimizing a ridge approximation over $\ma U$. 
However, the recovery problem is similar in spirit to the ridge approximation problem. 

\section{Separable reformulation\label{sec:reformulate}}
The key to efficiently constructing a data-driven polynomial ridge approximation
is exploiting structure to reduce the effective number of parameters.
As stated in~\cref{eq:pra_opt}, this approximation problem requires minimizing
over two sets of variables: the polynomial $g$ and the subspace spanned by $\ma U$.
Defining $g$ through its expansion in a basis $\lbrace \psi_j \rbrace_{j=1}^N$
of polynomials of total degree $p$ on $\R^n$, $\mathbb{P}^p(\R^n)$, 
we write $g$ as the sum
\begin{equation}
	g(\ve y) := \sum_{j=1}^N c_j \psi_j(\ve y), \qquad \ve y \in \R^n, \qquad N := {n + p \choose p}
\end{equation}
where the coefficients $c_j$ are entries of a vector $\ve c\in \R^N$ specifying the polynomial.
Using this expeansion, evaluating the ridge approximation at the inputs $\lbrace \ve x_i\rbrace_{i=1}^M$ 
is equivalent to the product of a Vandermonde-like matrix $\ma V(\ma U)$ and the coefficients~$\ve c$:
\begin{equation}\label{eq:defV}
	g(\ma U^\trans \ve x_i) = [\ma V(\ma U)\ve c]_i, \qquad
		[\ma V(\ma U)]_{i,j} := \psi_j(\ma U^\trans \ve x_i),
		\qquad \ma V: \R^{m\times n} \to \R^{M \times N};
\end{equation}
where here $[\ma A]_{i,j}$ denotes the $i$th row and $j$th column of $\ma A$.
The matrix-vector product~\cref{eq:defV} allows us to restate 
the polynomial ridge approximation problem~\cref{eq:pra_opt} in terms of coefficients $\ve c\in \R^N$
rather than the polynomial $g \in \mathbb{P}^p(\R^n)$:
\begin{equation}\label{eq:min_equiv}
	\underset{\substack{g \in \mathbb{P}^p(\R^n) \\
				 \Range \ma U \in \mathbb{G}(n, \R^m)}}{\operatorname{minimize}}
		\ 
		\sum_{i=1}^M \left[ f(\ve x_i) - g(\ma U^\trans \ve x_i)\right]^2
		\quad \Longleftrightarrow
		\underset{\substack{\ve c\in \R^{N} \\ \Range \ma U \in \mathbb{G}(n,\R^m)}}{\operatorname{minimize}}
			\left\|
				\ve f - \ma V(\ma U)\ve c
			\right\|_2^2,
		\ \ 
\end{equation}
where $\ve f\in \R^M$ holds the values of $f$ at $\ve x_i$, $[\ve f]_i := f(\ve x_i)$.
This new formulation reveals that the polynomial ridge approximation problem is a 
\emph{separable nonlinear least squares problem}, as for a fixed $\ma U$, 
$\ve c$ is easily found by solving a least squares problem.
This structure allows us to use \emph{variable projection}~\cite{GP73}
to define an equivalent optimization problem over $\ma U$ alone 
and construct its Jacobian as described in \cref{sec:reformulate:varpro}.
Then in \cref{sec:reformulate:orthogonality} we prove 
that slices of the Jacobian are orthogonal to the range of $\ma U$,
allowing us to reduce the cost of computing the Grassmann-Gauss-Newton step given in \cref{sec:algorithm}.
However, first we address the choice of polynomial basis $\lbrace \psi_j \rbrace_{j=1}^N$
as this choice has a profound influence on the conditioning of the approximation problem.

\subsection{Choice of basis}
Polynomial approximation has a well deserved reputation for being ill-conditioned~\cite[Ch.~22]{Hig02}.
For example, for a one-dimensional ridge function constructed in the monomial basis 
the matrix $\ma V(\ma U)$ is a Vandermonde matrix,
\begin{equation}
	\ma V(\ma U) =
	\begin{bmatrix}
		1 & (\ma U^\trans \ve x_1) & \cdots & (\ma U^\trans \ve x_1)^p \\
		1 & (\ma U^\trans \ve x_2) & \cdots & (\ma U^\trans \ve x_2)^p \\
		\vdots & \vdots & & \vdots \\
		1 & (\ma U^\trans \ve x_M) & \cdots & (\ma U^\trans \ve x_M)^p 
	\end{bmatrix}
	\!\! \in \R^{M \times (p+1)},
	\ \ n = 1,
	\ \ \psi_j(y) = y^{j-1}.
\end{equation}
Unless the sample points $\ve y_i = \ma U^\trans \ve x_i \in \R$ are
uniformly distributed on the complex unit circle, 
the condition number of this matrix grows exponentially in polynomial degree $p$~\cite{Pan16}.
As we assume we are given the sample points,
our only hope for controlling the condition number comes from our choice of basis $\lbrace \psi_j \rbrace_{j=1}^N$.
Here we invoke two strategies to control this condition number:
shifting and scaling the projected points $\ve y_i$
and choosing an appropriate orthogonal basis $\lbrace \psi_j \rbrace_{j=1}^N$.

We are free to shift and scale the projected points $\ve y_i$ as
any polynomial basis of total degree $p$ is still a basis for this space
when composed with an affine transformation.
Namely, if $\ve \eta: \R^n \to\R^n$ is the affine transformation 
\begin{equation}
	\ve \eta(\ve y) := \ve a + \ma D \ve y
\end{equation} 
and $\lbrace \psi_j \rbrace_{j=1}^N$ is a basis for $\mathbb{P}^p(\R^n)$, 
then $\lbrace \psi_j \circ \ve \eta \rbrace_{j=1}^N$ is also a basis for $\mathbb{P}^p(\R^n)$.
By careful choice of $\ve\eta$ we can significantly decrease the condition number.
For example as shown on the left of \cref{fig:conditioning},
shifting and scaling the projected points to the interval $[-1,1]$
drastically decreases the condition number in the monomial basis.
Similar results are seen for the other bases.
 
\begin{figure}
\centering
\noindent
% Following centered legend example from  https://tex.stackexchange.com/questions/315224/center-legend-above-or-below-a-groupplot-without-references
\begin{tikzpicture}
	\begin{groupplot}[
		group style = {group size = 2 by 1, horizontal sep = 4.5em},
		width = 6.5cm,
		height = 5cm,
		ymode = log,
		ymax = 1e20,
		ymin = 1,
		xmax = 30.5,
		xtick = {1,5,10,15,20,25,30},
		ytick = {1e0,1e5,1e10,1e15,1e20,1e25,1e30},
		xlabel = polynomial degree,
		ylabel = {condition number of $\ma V(\ma U)$},
		legend style = {/tikz/every even column/.append style={column sep=0.5cm}}, 
		title style = {yshift = -1ex},
		]
		\nextgroupplot[title = { $n=1$, $\ma U = \ve 1/10$}]
		\coordinate (c1) at (rel axis cs:0,1);
		\addplot[colorbrewerA5, only marks, mark = Mercedes star, semithick] table [x=degree, y=monomial_False] {fig_conditioning_1.dat};
		\addplot[colorbrewerA3, only marks, mark = Mercedes star flipped, semithick] table [x=degree, y=monomial_True] {fig_conditioning_1.dat};
		\addplot[colorbrewerA1, only marks, mark = +, semithick] table [x=degree, y=legendre_True] {fig_conditioning_1.dat};
		%\addplot[colorbrewerA2, only marks, mark = x] table [x=degree, y=chebyshev_True] {fig_conditioning_1.dat};
		\addplot[colorbrewerA2, only marks, mark = x, semithick] table [x=degree, y=hermite_True] {fig_conditioning_1.dat};

		\nextgroupplot[title = { $n=2$, $\ma U$ random},
			legend style={at={($(0,0)+(1cm,1cm)$)},legend columns=5,fill=none,draw=black,anchor=center,align=center},
			legend to name=fred
		]
		\coordinate (c2) at (rel axis cs:1,1);
		\addplot[colorbrewerA5, only marks, mark = Mercedes star, semithick] table [x=degree, y=monomial_False] {fig_conditioning_2.dat};
		\addplot[colorbrewerA3, only marks, mark = Mercedes star flipped, semithick] table [x=degree, y=monomial_True] {fig_conditioning_2.dat};
		\addplot[colorbrewerA1, only marks, mark = +, semithick] table [x=degree, y=legendre_True] {fig_conditioning_2.dat};
		%\addplot[colorbrewerA2, only marks, mark = x] table [x=degree, y=chebyshev_True] {fig_conditioning_2.dat};
		\addplot[colorbrewerA2, only marks, mark = x, semithick] table [x=degree, y=hermite_True] {fig_conditioning_2.dat};
		%\addplot[colorbrewerA4, only marks, mark = o, semithick, mark size = 1.5pt] table [x=degree, y=orthogonal_True] {fig_conditioning_2.dat};
		 
		\addlegendentry{monomial, unscaled};
		\addlegendentry{monomial, scaled to $[-1,1]^n$}
		\addlegendentry{Legendre}
		%\addlegendentry{Chebyshev}
		\addlegendentry{Hermite}    

	\end{groupplot}	
	\coordinate (c3) at ($(c1)!.5!(c2)$);
	\node[below] at (c3 |- current bounding box.south) {\pgfplotslegendfromname{fred}};
\end{tikzpicture}
\caption{The condition number of the matrix $\ma V(\ma U) \in \R^{M\times N}$
	based on different polynomial bases generated from $M = 1000$ random samples drawn with uniform probability from
	$\set D = [0,1]^{100}$.
	On the left, with the subspace chosen to equally weight each component, 
	we might expect the Hermite basis to be well conditioned 
	as low-dimensional projections of high-dimensional data are often Gaussian~\cite{Diaconis1984}.
	On the right we show an example with a two-dimensional subspace drawn randomly with uniform probability on the sphere.
}
\label{fig:conditioning}
\end{figure}
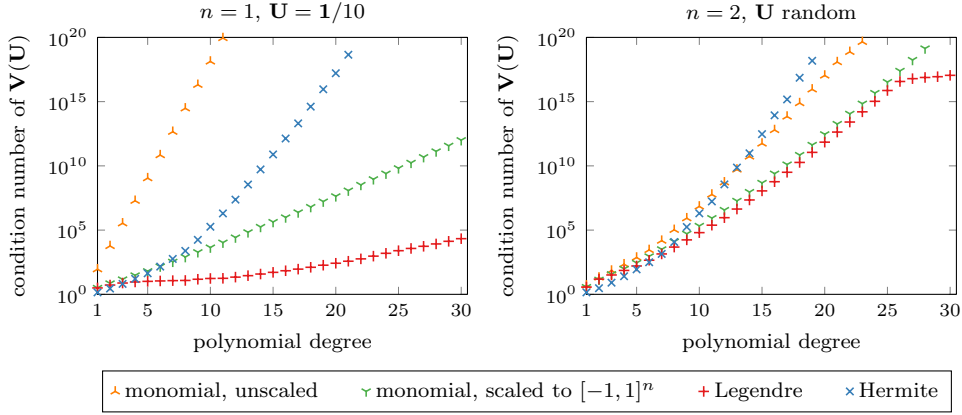 

The second strategy for controlling the condition number is to change the polynomial basis.
Ideally, we would choose a basis orthonormal under the weighted inner-product
induced by the projected points,
\begin{equation}\label{eq:basis_orth}
	\int \psi_j(\ve y) \psi_k(\ve y) \D \mu(\ve y) = \! \sum_{i=1}^M \psi_j(\ve y_i) \psi_k(\ve y_i)
		= \begin{cases} 
			1, & j = k; \\
			0, & j \ne k;
		\end{cases}
	\quad \mu(\ve y) = \sum_{i=1}^M \delta(\ve y - \ve y_i).
\end{equation}
With this choice, $\ma V(\ma U)$ would have a condition number of one.
However, constructing this basis for an arbitrary set of points is fraught
with the same ill-conditioning issues as $\ma V(\ma U)$ itself.
Instead, a reasonable heuristic is to choose a basis $\lbrace \psi_j \rbrace_{j=1}^M$
which is orthogonal under a weight $\mu$ that approximates the distribution of $\lbrace \ve y_i \rbrace_{i=1}^M$.
For example in one dimension, the Legendre polynomials are orthogonal with respect to the uniform measure on $[-1,1]$
and the Hermite polynomials are orthogonal with respect to $\mu(y) = e^{-y^2/2}$.
In \cref{fig:conditioning},
we compare the conditioning of three different bases applied to
one thousand points $\ve x_i$ uniformly randomly chosen from the cube $[0,1]^{100}$
and projected onto a one-dimensional subspace spanned by the ones vector.
By the law of large numbers, the projected points $\lbrace \ma U^\trans \ve x_i \rbrace_{i=1}^M$
are approximately normally distributed.
After shifting and scaling these points to have zero mean and unit standard deviation, 
$\lbrace \ve y_i = \ma U^\trans \ve x_i\rbrace_{i=1}^M$ will approximately sample the normal distribution
and hence we might expect the Hermite basis to be well-conditioned.
Although the Hermite basis is well conditioned for low degree polynomials,
the condition number rapidly grows for high degree polynomials
which is consistent with the observations by Hampton and Doostan~\cite{Hampton2015}.
In contrast, the Legendre polynomial basis when scaled and shifted to $[-1,1]$
provides a well conditioned basis even when the degree is high.

When we seek a ridge approximation with two or more dimensions, 
we can form the basis $\lbrace \psi_j \rbrace_{j=1}^M$ for total degree polynomials in $\R^n$
from a tensor product of one dimensional polynomials.
If this one dimensional basis $\lbrace \varphi_k \rbrace_{k=0}^p \subset \mathbb{P}^p(\R)$
has the property that $\varphi_k$ is degree $k$, then the basis $\lbrace \psi_j\rbrace_{j=1}^N$ has elements
\begin{equation}
	\psi_j(\ve y) = \prod_{k=1}^n \varphi_{[\ve \alpha_j]_k}([\ve y]_k),
	\qquad \ve \alpha_j \in \mathbb{N}^n, \quad |\ve \alpha_j| :=\sum_{k=1}^n [\ve \alpha_j]_k \le p
\end{equation} 
where $\lbrace\ve\alpha_j\rbrace_{j=1}^N$ is an enumeration of the multi-indices satisfying $|\ve \alpha_j| \le p$.
For the Legendre basis, this tensor product is still an orthogonal basis with respect to the uniform measure on $[-1,1]^n$.
However, the condition number of the matrix $\ma V(\ma U)$ built from the Legendre basis,
as well as the other bases, grows rapidly despite the scaling and shifting.
Since the tensor product Legendre basis is least ill-conditioned, 
we use this basis in the remainder of this paper.

\subsection{Variable projection\label{sec:reformulate:varpro}}
The key insight of variable projection is that if the nonlinear parameters $\ma U$
are fixed, the linear parameters $\ve c$ are easily recovered using the Moore-Penrose pseudoinverse, denoted by $^+$:
\begin{equation}
	\underset{\ve c}{\operatorname{minimize}} \ \| \ve f - \ma V(\ma U)\ve c\|_2 
		\quad \Rightarrow  \quad
		\ve c = \ma V(\ma U)^+\ve f.
\end{equation}
Then, replacing $\ve c$ with $\ma V(\ma U)^+\ve f$ in \cref{eq:min_equiv}
reveals an optimization problem over $\ma U$ alone:
\begin{equation}
	\underset{\Range \ma U \in \mathbb{G}(n, \R^m)}{\operatorname{minimize}}\  \| \ve f - \ma V(\ma U) \ma V(\ma U)^+\ve f\|_2^2.
\end{equation}
Recognizing $\ma V(\ma U)\ma V(\ma U)^+$ as an orthogonal projector onto the range of $\ma V(\ma U)$, 
we can rewrite the minimization problem as 
\begin{equation}
	\underset{\Range \ma U \in \mathbb{G}(n, \R^m)}{\operatorname{minimize}}\, \| \ma P_{\ma V(\ma U)}^\perp \ve f\|_2^2
\end{equation}
where $\ma P_{\ma V(\ma U)}^\perp$ is the orthogonal projector onto the complement of the range of $\ma V(\ma U)$.

Golub and Pereyra provide a formula for the derivative of the residual $\ve r(\ma U) := \ma P_{\ma V(\ma U)}^\perp \ve f$~\cite[eq.~(5.4)]{GP73}
with respect to $\ma U$.
Denoting the derivative of the residual as a tensor 
$\ten J(\ma U)\in \R^{M \times m\times n}$
where $[\ten J(\ma U)]_{i,j,k} = \partial [\ve r(\ma U)]_{i}/\partial [\ma U]_{j,k}$,
then
\begin{equation}\label{eq:ten_J}
	[\ten J(\ma U)]_{\cdot, j,k} = -\left[
			\left(\ma P_{\ma V(\ma U)}^\perp \frac{\partial \ma V}{\partial [\ma U]_{j,k}} (\ma U) \  \ma V(\ma U)^-\!\right)
			+ 
			\left(\ma P_{\ma V(\ma U)}^\perp \frac{\partial \ma V}{\partial [\ma U]_{j,k}} (\ma U) \ \ma V(\ma U)^-\!\right)^{\!\!\!\trans\ }
		\right]\ve f,
\end{equation}
where $^-$ denotes a \emph{least squares pseudoinverse} such that 
$\ma V\ma V^-\ma V = \ma V$ and $\ma V\ma V^- = (\ma V\ma V^-)^\trans$,
a weaker pseudo-inverse than the Moore-Penrose pseudo-inverse~\cite[Ch.~6]{CM79}.
Then, from the definition of $\ma V(\ma U)$ in~\cref{eq:defV}
its derivative is:
\begin{align}
	\left[\frac{\partial \ma V}{\partial [\ma U]_{k,\ell}}(\ma U)\right]_{i,j} &= 
		[\ve x_i]_k \left. \frac{\partial \psi_j(\ve y)}{\partial [\ve y]_\ell} \right|_{\ve y = \ma U^\trans \ve x_i}.
\end{align}
In particular, for a tensor product basis
$\lbrace \psi_j(\ve y) = \prod_{k=1}^n \varphi_{[\ve \alpha_j]_k}( \ve \eta([\ve y]_k)) \rbrace_{j=1}^{N}$
composed with an affine transformation
$\ve \eta(\ve y) = \ve a + \diag(\ve d) \ve y$,
we have
\begin{align}
	\left[\frac{\partial \ma V}{\partial [\ma U]_{k,\ell}}(\ma U)\right]_{i,j} &= 
		 [\ve d]_k \; [\ve x_i]_k \; \varphi_{[\ve \alpha_j]_\ell}'([\ve \eta(\ma U^\trans \ve x_i)]_\ell)\;
			 \prod_{\substack{q=1\\ q\ne \ell}}^n \; \varphi_{[\ve \alpha_j]_q}([\ve \eta(\ma U^\trans \ve x_i)]_q)\ .
\end{align}

\subsection{Orthogonality\label{sec:reformulate:orthogonality}}
In addition to constructing an explicit formula for the Jacobian of $\ve r(\ma U)$, 
we also show that slices of this Jacobian $\ten J(\ma U)$ are orthogonal to $\ma U$;
that is $\ma U^\trans [\ten J(\ma U)]_{i,\cdot,\cdot} = \ma 0$.
This emerges due to the use of variable projection and
the use of a polynomial basis for the columns of $\ma V(\ma U)$
so that when $\ma U^\trans$ is multiplied by slices of the Jacobian, 
this derivative is mapped back into the range of the polynomial basis for $\ma V(\ma U)$.
Later, in \cref{sec:algorithm:gn},
we exploit this structure to reduce the computational burden during optimization. 

\begin{theorem}\label{thm:orth}
	If $\ten J(\ma U)$ is defined as in~\cref{eq:ten_J}, 
	then $\ma U^\trans [\ten J(\ma U)]_{i,\cdot,\cdot} = \ma 0$ $\forall i$.
\end{theorem}
\begin{proof}
First note,
\begin{align*}
	\left[\ma U^\trans [\ten J(\ma U)]_{i,\cdot,\cdot}\right]_{j,k}\!\!
		&= \sum_{\ell=1}^m [\ma U]_{\ell,j} [\ten J(\ma U)]_{i,\ell,k} \\
		&= \!-\!\sum_{\ell=1}^m [\ma U]_{\ell,j} \! \!\left[  
			\ma P_{\ma V(\ma U)}^\perp \frac{\partial \ma V(\ma U)}{\partial [\ma U]_{\ell,k}}
				\ma V(\ma U)^{-} \ve f
			\!+\! \ma V(\ma U)^{-\!\trans} \frac{\partial \ma V(\ma U) }{\partial [\ma U]_{\ell,k}}^\trans \! \ma P_{\ma V(\ma U)}^\perp \ve f
		\right]_i\!\!.
\end{align*}
By linearity, we take the sum inside each term, leaving
\begin{equation}\label{eq:product}
		= -\!\left[  
			\ma P_{\ma V(\ma U)}^\perp \! \left[ \sum_{\ell=1}^m [\ma U]_{\ell, j} \frac{\partial \ma V(\ma U)}{\partial [\ma U]_{\ell,k}} \right]
				\! \ma V(\ma U)^{-} \ve f
			+ \ma V(\ma U)^{-\!\trans} \!
				\left[ \sum_{\ell=1}^m [\ma U]_{\ell, j} \frac{\partial \ma V(\ma U) }{\partial [\ma U]_{\ell,k}}^\trans\right]
		 \!\ma P_{\ma V(\ma U)}^\perp \ve f
		\right]_i\!.
\end{equation}
Defining this interior sum as the matrix $\ma W^{(j,k)}$,
\begin{equation}
	\ma W^{(j,k)} := \sum_{\ell=1}^m [\ma U]_{\ell,j} \frac{\partial \ma V(\ma U)}{\partial [\ma U]_{\ell,k}}
\end{equation}
we can show $\ma W^{(j,k)}$ is in the range of $\ma V(\ma U)$.
Without loss of generality, we work in the
monomial tensor product basis where $\psi_j(\ve y) = \prod_{k=1}^n [\ve y]_k^{[\ve\alpha_j]_k}$
where
\begin{align}
	[\ma V(\ma U)]_{i,j} &= \prod_{k=1}^n [\ma U^\trans \ve x_i]_k^{[\ve \alpha_j]_k}, \\
	\left[ 
		\frac{\partial \ma V}{\partial [\ma U]_{k,\ell}}(\ma U)
	\right]_{i,j}
	&= [\ve\alpha_j]_\ell \ [\ve x_i]_k \  [\ma U^\trans \ve x_i]^{[\ve \alpha_j]_{\ell} - 1}_\ell 
		\prod_{\substack{q = 1\\ q\ne \ell}}^n [\ma U^\trans \ve x_i]_q^{[\ve \alpha_j]_q}.
\end{align}
Then the interior sum encodes the inner product $[\ma U^\trans \ve x_r]_j$,
\begin{align*}
	[\ma W^{(j,k)}]_{r,q} = \left[\sum_\ell [\ma U]_{\ell,j} \frac{\partial \ma V(\ma U)}{\partial [\ma U]_{\ell,k}}\right]_{r,q}
		\!\!\!\!
		&= \sum_\ell [\ma U]_{\ell,j} [\ve \alpha_q]_k [\ma U^\trans\ve x_r]_k^{[\ve \alpha_q]_k -1} [\ve x_r]_\ell
			\prod_{\substack{s=1\\ s\ne k}}^n [\ma U^\trans \ve x_r]^{[\ve \alpha_q]_s}_s \\
		&= [\ve \alpha_q]_k [\ma U^\trans \ve x_r]^{[\ve \alpha_q]_k - 1}_k [\ma U^\trans \ve x_r]_j
			\prod_{\substack{s=1\\ s\ne k}}^n [\ma U^\trans \ve x_r]^{[\ve \alpha_q]_s}_s. 
\end{align*}
This term on right is a polynomial of total degree at most $p$ in $\ma U^\trans \ve x_r$,
since although the power on the $k$th term of $\ma U^\trans \ve x_r$ has decreased by one,
the $j$th term has increased by one, leaving the total degree of this term the same,
namely less than or equal to $p$.
Thus, as $\ma W^{(j,k)}\in \Range(\ma V(\ma U))$ then $\ma P_{\ma V(\ma U)}^\perp \ma W^{(j,k)} = \ma 0$.
As this product appears in both terms of \cref{eq:product}, we conclude
$\ma U^\trans [\ten J]_{i,\cdot,\cdot} = \ma 0$.
\end{proof}

\section{Optimization on the Grassmann manifold\label{sec:algorithm}}
Having implicitly found the polynomial $g$ using variable projection
in the previous section, 
we now develop an algorithm for solving the ridge approximation problem posed over $\ma U$ alone:
\begin{equation}
	\underset{\Range \ma U \in \mathbb{G}(n,\R^m)}{\operatorname{minimize}} \  \| \ma P_{\ma V(\ma U)}^\perp \ve f\|_2^2.
\end{equation}
This optimization problem over the Grassmann manifold of all $n$-dimensional subspaces of $\R^m$
is more complicated than optimization on Euclidean space.
Here we follow the approach of Edelman, Arias, and Smith~\cite{EAS98}
where the subspace is parameterized by a matrix $\ma U\in \R^{m\times n}$
with orthonormal columns, i.e., $\ma U$ satisfies the constraint $\ma U^\trans \ma U = \ma I$.
We first review Newton's method on the Grassmann manifold following their construction
before modifying their approach to construct a Gauss-Newton method on the Grassmann manifold
for the data-driven polynomial ridge approximation problem~\cref{eq:pra_opt}.
In the process we note that the orthogonality of $\ma U$ to slices of the Jacobian
as proved in \cref{thm:orth} allows many terms to drop, simplifying the optimization problem.
An alternative approach would be to follow the Gauss-Newton approach of Absil, Mahony, and Sepulchre~\cite[\S 8.4.1]{AMS08}
which removes the orthogonality constraint.
However, by working in the framework of Edelman, Arias, and Smith
we are able to show the additional terms in the Hessian drop due to the orthogonality 
result from \cref{thm:orth}.

\subsection{Newton's method on the Grassmann manifold}
To begin, we first review Newton's method on the Grassmann manifold,
following Edelman, Arias, and Smith~\cite{EAS98}.
There are two key properties we consider:
how to update $\ma U$ given a search direction $\ma \Delta$
and how to choose the search direction $\ma \Delta$ using Newton's method.

When optimizing in a Euclidean space, given a search direction $\ve d$ and an initial point $\ve x_0$,
the next iterate is chosen along the trajectory $\ve x(t) = \ve x_0 + t \ve d$ for $t\in(0,\infty)$
where $t$ is selected to ensure convergence.
However, if we were to apply the same search strategy to our parameterization of the Grassmann manifold
this would result in a point that does not obey the orthogonality constraint: $\ma U^\trans \ma U = \ma I$.
Instead, we replace the linear trajectory with a geodesic (a contour with constant derivative).
Following~\cite[Thm.~2.3]{EAS98}, if $\ma \Delta\in \R^{m\times n}$ 
is the search direction that is tangent to the current estimate $\ma U_0$, $\ma U_0^\trans \ma \Delta = \ma 0$,
then we choose $\ma U(t)$ on the geodesic
\begin{equation}\label{eq:geodesic}
	\ma U(t) = \ma U_0 \ma Z \cos (\ma \Sigma t) \ma Z^\trans + \ma Y \sin(\ma \Sigma t) \ma Z^\trans 
\end{equation}
where $\ma \Delta = \ma Y \ma \Sigma \ma Z^\trans $ is the short form singular value decomposition.

In addition to changing the search trajectories, 
optimization on the Grassmann manifold changes the gradient and Hessian.
Here we consider an arbitrary function $\phi$ of a matrix $\ma U\in \R^{m\times n}$ with orthonormal columns.
Following~\cite[(2.52)\&(2.56)]{EAS98},
the first and second derivatives of $\phi$ with respect to the entries of $\ma U$ are
\begin{align}
	[\phi_{\ma U}]_{i,j} &:= \frac{\partial \phi}{\partial [\ma U]_{i,j}}, 
	& \phi_\ma U&\in \R^{m\times n}; %\label{eq:phi1}
	\\
	[\phi_{\ma U\ma U}]_{i,j,k,\ell} &:= \frac{\partial^2 \phi}{\partial [\ma U]_{i,j} \partial [\ma U]_{k,\ell}}, 
	& \phi_{\ma U \ma U}&\in \R^{m \times n \times m \times n}. 
	%\label{eq:phi2}
\end{align} 
With these definitions, the gradient of $\phi$ on the Grassmann manifold that is tangent to $\ma U$ is~\cite[(2.70)]{EAS98}:
\begin{align}%\label{eq:grad}
	\grad \phi = \phi_\ma U - \ma U \ma U^\trans \phi_\ma U = \ma P_{\ma U}^\perp \phi_\ma U \in \R^{m\times n}.
\end{align}
The Hessian of $\phi$ on the Grassmann manifold, $\Hess \phi \in \R^{m\times n\times m \times n}$,
is defined by its action on two test matrices $\ma \Delta, \ma X\in \R^{m\times n}$~\cite[(2.71)]{EAS98}
\begin{equation}\label{eq:Hessian}
	\Hess \phi (\ma \Delta, \ma X) = 
		\sum_{i,j,k,\ell} [\phi_{\ma U\ma U}]_{i,j,k,\ell} [\ma \Delta]_{i,j} [\ma X]_{k,\ell}
			- \trace (\ma \Delta^\trans \ma X \ma U^\trans \phi_\ma U).
\end{equation}
Using these definitions, 
Newton's method on the Grassmann manifold at $\ma U$ chooses 
the tangent search direction $\ma \Delta \in \R^{m\times n}$ satisfying~\cite[(2.58)]{EAS98} 
\begin{equation}\label{eq:newton_step}
		\ma U^\trans \ma \Delta = \ma 0
		\quad 
		\text{and}
		\quad 
		\Hess \phi(\ma \Delta, \ma X) = - \, \langle \grad \phi,\ma X\rangle 
		\quad  \forall \ \ma X \text{ such that } \ma U^\trans\ma X = 0 
\end{equation} 
where the inner product $\langle \cdot, \cdot\rangle$ on this space is~\cite[(2.69)]{EAS98}
\begin{equation}\label{eq:inner}
	\langle \ma X, \ma Y\rangle = \trace \ma X^\trans \ma Y = (\vectorize \ma X)^\trans (\vectorize \ma Y)
\end{equation}
and $\vectorize$ maps the matrix $\R^{m\times n}$ into the vector $\R^{mn}$.

\subsection{Grassmann Gauss-Newton\label{sec:algorithm:gn}}
The challenge with Newton's method is it requires second derivative information
which is difficult to obtain for our problem.
Here we replace the true Hessian with the Gauss-Newton approximation, 
yielding a Grassmann Gauss-Newton method.
To summarize the key points of the following argument, 
the orthogonality of slices of the Jacobian to $\ma U$
established in \cref{thm:orth} allows us to drop the second term in the Hessian~\cref{eq:Hessian}
and replace the normal equations~\cref{eq:newton_step}
with a better conditioned least squares problem~\cref{eq:gn_step}
analogous to the Gauss-Newton method in Euclidean space~\cite[(10.26)]{NW06}.
A similar Gauss-Newton method is given in~\cite[\S8.4]{AMS08}
where the subspace is parameterized by a matrix that is not necessarily orthogonal
and uses a different geodesic step.

The objective function for data-driven polynomial ridge approximation is:
\begin{equation}
	\phi(\ma U) = \frac12 \| \ma P_{\ma V(\ma U)}^\perp \ve f\|_2^2 = \frac12 \|\ve r(\ma U)\|_2^2
\end{equation}
where first and second derivatives of $\phi$ are  
\begin{align}
	\phi_\ma U &= \sum_{i=1}^M [\ten J(\ma U)]_{i,\cdot,\cdot} [\ve r(\ma U)]_i; \\
	[\phi_{\ma U\ma U}]_{i,j,k,\ell} &= \sum_{q} [\ten J(\ma U)]_{q,i,j} [\ten J(\ma U)]_{q,k,\ell}
		+ \sum_q [\ve r(\ma U)]_q \frac{[\ve r(\ma U)]_q}{\partial [\ma U]_{i,j} \partial [\ma U]_{k,\ell}} .
\end{align}
Invoking the Gauss-Newton approximation, we drop the second term above %of \cref{eq:pra_Hess}
\begin{equation}
	[\phi_{\ma U\ma U}]_{i,j,k,\ell} \approx [\widetilde{\phi}_{\ma U\ma U}]_{i,j,k,\ell}
		:= \sum_{q} [\ten J(\ma U)]_{q,i,j} [\ten J(\ma U)]_{q,k,\ell}.
\end{equation}
Then replacing $\phi_{\ma U\ma U}$ with $\widetilde{\phi}_{\ma U\ma U}$ in the Hessian \cref{eq:Hessian} yields
the approximate Hessian
\begin{equation}\label{eq:hess_approx}
	\widetilde{\Hess}\ \phi (\ma \Delta, \ma X) := 
		\sum_{i,j,k,\ell, q} [\ten J(\ma U)]_{q,i,j} [\ma \Delta]_{i,j} [\ten J(\ma U)]_{q,k,\ell} [\ma X]_{k,\ell}
			- \trace (\ma \Delta^\trans \ma X \ma U^\trans \phi_\ma U).
\end{equation}
Immediately, we note that by \cref{thm:orth} $\ma U^\trans [\ten J(\ma U)]_{i,\cdot,\cdot} = \ma 0$ and hence,
\begin{equation}\label{eq:phi_orth}
	\ma U^\trans \phi_\ma U = \sum_i \ma U^\trans [\ten J(\ma U)]_{i,\cdot,\cdot} [\ve r(\ma U)]_i
		= \ma 0.
\end{equation}
Thus the second term drops out of the approximate Hessian~\cref{eq:hess_approx}, leaving
\begin{equation}
	\widetilde{\Hess}\  \phi (\ma \Delta, \ma X) = 
		\sum_{i,j,k,\ell, q} [\ten J(\ma U)]_{q,i,j} [\ma \Delta]_{i,j} [\ten J(\ma U)]_{q,k,\ell} [\ma X]_{k,\ell}.
\end{equation}
This summation can be rearranged to look like the more familiar Hessian approximation $\ma J^\trans \ma J$
when the Jacobian $\ma J$ is a matrix.
Defining the vectorization operator for tensors that maps $\ten J(\ma U)\in \R^{M\times m\times n}$
to a matrix in $\R^{M\times mn}$,
this product above is
\begin{equation}
	\widetilde{\Hess}\ \phi(\ma \Delta, \ma X) = 
		(\vectorize \ma X)^\trans  (\vectorize \ten J(\ma U))^\trans
		(\vectorize \ten J(\ma U)) (\vectorize \ma \Delta).
\end{equation}

With this familiar expression for the approximate Hessian, 
we now seek to rework the Newton step  
from a square linear system into an overdetermined least squares problem.
First, invoking \cref{thm:orth} via \cref{eq:phi_orth}, we note that 
the gradient automatically satisfies the orthogonality constraint since $\ma U^\trans \phi_\ma U=\ma 0$:
\begin{equation}\label{eq:pra_grad}
	\grad \phi(\ma U) = \phi_\ma U - \ma U\ma U^\trans \phi_\ma U = \phi_\ma U 
	= (\vectorize \ten J(\ma U))^\trans \ve r(\ma U).
\end{equation}
Then, examining the right hand side of the Newton step~\cref{eq:newton_step},
\begin{align}
	\langle \grad \phi, \ma X\rangle% = \langle \phi_{\ma U}, \ma X \rangle  
	= (\vectorize \ma X)^\trans (\vectorize \ten J(\ma U))^\trans \ve r(\ma U).
	% \trace \phi_\ma U
	%&= \trace \phi_\ma U^\trans \ma X 
	%	= \trace \left[ \sum_i \ma X^\trans[\ten J(\ma U)]_{i,\cdot,\cdot} [\ve r(\ma U)]_i \right] \\
	%	&= \sum_i (\vectorize \ma X)^\trans [\vectorize \ten J(\ma U)]_{i,\cdot}^\trans
	%		[\ve r(\ma U)]_i  = (\vectorize \ma X)^\trans (\vectorize \ten J(\ma U))^\trans \ve r(\ma U)
\end{align}
Thus, the Gauss-Newton step $\ma \Delta \in \R^{m\times n}$ 
is the matrix satisfying $\ma U^\trans \ma \Delta = \ma 0$ 
and 
\begin{align}
	(\vectorize \ma X)^{\!\trans}\! (\vectorize \ten J(\ma U))^{\!\trans}\! (\vectorize \ten J(\ma U)) (\vectorize \ma \Delta)
	&= -(\vectorize \ma X)^{\!\trans} \!(\vectorize \ten J(\ma U))^{\!\trans} \ve r(\ma U), 
		%\forall \ma X \text{ s.t. }\!\! \ma U^{\!\trans} \!\ma X \!=\! \ma 0. 
\end{align}
for all test matrices $\ma X$ such that $\ma U^\trans \ma X = \ma 0$.
To convert this into a least squares problem, we first replace the constraint $\ma U^\trans \ma X = \ma 0$
by substituting $\ma X$ by $\ma P_{\ma U}^\perp\ma X$ where $\ma P_{\ma U}^\perp$ is the orthogonal projector
onto the complement of the range of $\ma U$, $\ma P_{\ma U}^\perp = \ma I - \ma U\ma U^\trans$.
This leaves a linear system of equations over all test matrices $\ma X\in \R^{m\times n}$:
\begin{align}
	(\vectorize \ma P_{\ma U}^\perp \ma X)^{\!\trans}\! (\vectorize \ten J(\ma U))^{\!\trans}\! (\vectorize \ten J(\ma U)) (\vectorize \ma \Delta)
	&\!=\! -(\vectorize \ma P_{\ma U}^\perp \ma X)^{\!\trans} \!(\vectorize \ten J(\ma U))^{\!\trans} \ve r(\ma U). 
\end{align}
On the left, the projector $\ma P_\ma U^\perp$ vanishes by \cref{thm:orth}
\begin{align}
	\left[ (\vectorize \ten J(\ma U))(\vectorize \ma P_\ma U^\perp \ma X)\right]_i
		&= \Tr \left[ (\ma P_{\ma U}^\perp \ma X)^\trans [\ten J(\ma U)]_{i,\cdot,\cdot} \right]
		= \Tr \left[ \ma X^\trans (\ma I - \ma U \ma U^\trans) [\ten J(\ma U)]_{i,\cdot,\cdot}\right] \nonumber \\
		&= \Tr \left[ \ma X^\trans [\ten J(\ma U)]_{i,\cdot,\cdot}\right]
		= (\vectorize \ten J(\ma U))(\vectorize \ma X). \label{eq:porth}
\end{align}
Then, using the coordinate matrices $\ve e_i \ve e_j^\trans$ as a basis for $\ma X \in \R^{m \times n}$,
we then recover the normal equations 
\begin{align}
	 (\vectorize \ten J(\ma U))^{\!\trans}\! (\vectorize \ten J(\ma U)) (\vectorize \ma \Delta)
	= (\vectorize \ten J(\ma U))^{\!\trans} \ve r(\ma U).
\end{align}
Hence as in a Euclidean space (c.f.~\cite[(10.26)]{NW06}),
the Gauss-Newton step $\ma \Delta$ is the solution to the linear least squares problem
\begin{equation}\label{eq:gn_step}
	\underset{\substack{\ma \Delta \in \R^{m\times n}\\ \ma U^\trans \ma \Delta = \ma 0}}{\operatorname{minimize}}
		\left\|
			\vectorize \ten J(\ma U) \vectorize \ma \Delta - \ve r(\ma U)\right\|_2^2.
\end{equation}

Finally, we note that $\vectorize \ten J(\ma U)$ has a nullspace such that
when~\cref{eq:gn_step} is solved using the pseudoinverse, 
the step $\ma \Delta$ will automatically satisfy the constraint $\ma U^\trans \ma \Delta = \ma 0$.
Using \cref{eq:porth} we can insert the projector $\ma P_\ma U^\perp$ into the Gauss-Newton step \cref{eq:gn_step}
\begin{equation}
	\vectorize \ten J(\ma U) \vectorize \ma \Delta
	= \vectorize \ten J(\ma U)\vectorize (\ma P_{\ma U}^\perp \ma \Delta)
	= \vectorize \ten J(\ma U) [ \ma I_n \otimes \ma P_{\ma U}^\perp ] \vectorize \ma \Delta
\end{equation}
where $\otimes$ is the Kronecker product.
Thus $\vectorize \ten J(\ma U)$ has a nullspace of dimension $n^2$
which contains $\ma I_n \otimes \ma U\ma U^\trans$
and hence if \cref{eq:gn_step} is solved via the pseudoinverse,
$\ma \Delta$ will  automatically obey the constraint $\ma U^\trans \ma \Delta = \ma 0$.
This reveals our Gauss-Newton step:
\begin{equation}\label{eq:delta}
	\vectorize \ma \Delta = -[\vectorize \ten J(\ma U)]^+ \ve r(\ma U)
		= -\begin{bmatrix} [\ten J(\ma U)]_{\cdot,\cdot,1} & \ldots & [\ten J(\ma U)]_{\cdot,\cdot, n}\end{bmatrix}^+
			\ve r(\ma U).
\end{equation}
This pseudoinverse solution has a similar asymptotic cost to the normal equations:
an $\order(M(mn)^2)$ operation SVD where $M> mn$ 
compared to an $\order((mn)^3)$ dense linear solve.
However the pseudoinverse solution is better conditioned, 
avoiding the squaring of the condition number in the normal equations~\cite[\S2.3.3]{Bjo96}. 

\subsection{Algorithm}
We now combine the Gauss-Newton step~\cref{eq:delta}
with backtracking along the geodesic~\cref{eq:geodesic}
to construct a convergent data-driven polynomial ridge approximation algorithm.
The complete algorithm is given in \cref{alg:main},
using a pseudo-inverse to construct the Gauss-Newton step as in \cref{eq:delta}.
We ensure convergence by inserting a check on line~\ref{alg:main:check}
to ensure $\ma \Delta$ is always a descent direction, and if not, replacing it with the negative gradient.
Then the sequence of $\ma \Delta$ is a \emph{gradient related sequence}~\cite[Def.~4.2.1]{AMS08}
and the iterates $\ma U$ converge to a stationary point where $\grad \phi(\ma U) = \ma 0$
by~\cite[Cor.~4.3.2]{AMS08} since the Grassmann manifold is compact~\cite[\S9]{Won67}.

\begin{algorithm}[t]
\begin{minipage}{\linewidth}
\begin{algorithm2e}[H]
	%\caption{Ridge Approximation using Variable Projection}
	%\label{alg:main}
	\Input{Sample points $\ma X \in \R^{M\times m}$; 
		function values $\ve f\in \R^M$, $[\ve f]_i = f([\ma X]_{i,\cdot})$;
		subspace dimension $n$; polynomial degree $p$ (if $p=1$, then $n=1$);\\
		step length reduction factor $\gamma \in (0,1)$;
		Armijo tolerance $\beta \in (0,1)$.
		}
	\Output{Active subspace $\ma U\in \R^{m\times n}$; 
		polynomial coefficients $\ve c\in \R^N$, $N = {{n + p \choose p}}$.}
	Sample entries of $\ma Z\in \R^{m\times n}$ from a normal distribution\;
	Compute short form QR $\ma U \ma R \leftarrow \ma Z$ \comment{\quad hence $\ma U$ uniformly samples $\mathbb{G}(n, \R^m)$}
	\Repeat{$\ma U$ converges}{
		Compute $\lbrace \ve y_i = \ma U^\trans\ve x_i\rbrace_{i=1}^M$ \;
		Construct affine transformation $\ve \eta$ of $\lbrace \ve y_i \rbrace_{i=1}^M$ to $[-1,1]^n$\;
		Build $\ma V(\ma U)$ using tensor product Legendre basis composed with $\ve \eta$ via \cref{eq:defV}: $\ma V \leftarrow \ma V(\ma U)$\;
		Compute polynomial coefficients $\ve c \leftarrow \ma V^+ \ve f$\;
		Compute the residual: $\ve r \leftarrow \ve r(\ma U) = \ve f - \ma V \ve c$\;
		Build the Jacobian~\cref{eq:ten_J}: $\ten J \leftarrow \ten J(\ma U) \in \R^{M\times m \times n}$\;
		Build the gradient~\cref{eq:pra_grad}: $\ma G \leftarrow \ma G(\ma U) = \sum_{i=1}^M [\ten J]_{i,\cdot,\cdot} [\ve r]_i$\;
		Compute the short form SVD: $\ma Y \ma \Sigma \ma Z^\trans \leftarrow \vectorize \ten J$\;
		Compute the Gauss-Newton step~\cref{eq:delta}:
			$\!\!\vectorize\! \ma \Delta \!\!\leftarrow\! -\!\! \sum_{i=1}^{mn\!-n^{\!2}} \![\ma \Sigma]_{i,i}^{-1}\! 
				[\ma Z]_{\cdot, i}[\ma Y^{\!\trans}\! \ve r]_{i}$\;
		Compute slope along Gauss-Newton step: $\alpha \leftarrow \trace \ma G^\trans \ma \Delta = (\vectorize \ma G)^\trans(\vectorize \ma \Delta)$\;
		\If(\CommentSty{\quad $\ma \Delta$ is not a descent direction}){$\alpha \ge 0$}{\label{alg:main:check}
			$\ma \Delta \leftarrow -\ma G$\;
			$\alpha \leftarrow \trace \ma G^\trans \ma \Delta$\;
		}
		Compute the short form SVD: $\ma Y \ma \Sigma \ma Z^\trans \leftarrow \ma \Delta$\;
		\For(\CommentSty{\quad backtracking line search}){$t=\gamma^0, \gamma^1, \gamma^2, \ldots$}{
			Compute new step~\cref{eq:geodesic}: 
				$\ma U_+ \leftarrow \ma U \ma Z \cos (\ma \Sigma t) \ma Z^\trans + \ma Y \sin(\ma \Sigma t) \ma Z^\trans$\;
			Compute new residual: $\ve r_+ \leftarrow \ve f -  \ma V(\ma U_+)\ma V(\ma U_+)^+\ve f$\;
			\If(\CommentSty{\quad Armijo condition satisfied}){$\| \ve r_+\|_2 \le \| \ve r\|_2 + \alpha \beta t $}{
				{\bf break}\;
			}
		}	
		Update the estimate $\ma U \leftarrow \ma U_+$\;
	}
\end{algorithm2e}
\vspace{-12pt}
\end{minipage}
\caption{Variable projection polynomial ridge approximation}
\label{alg:main}
\end{algorithm}

When working with a relatively high order polynomial basis,
we have found a small Armijo tolerance $\beta$, such as $\beta=10^{-6}$, is necessary.
As the polynomial degree increases, the rapid oscillation of these polynomials
causes the gradient to grow rapidly and without a small tolerance, all but the tiniest steps are rejected.

In addition, our implementation equips this algorithm with three different convergence criteria:
small change between $\ma U$ and $\ma U_+$ as measured by the smallest canonical subspace angle,
small change in the norm of the residual, and
small gradient norm.
Further, the algorithm places a feasibility constraint on the polynomial degree and subspace dimension.
A linear polynomial ($p=1$) with any dimensional subspace is equivalent to a ridge function on a one-dimensional subspace;
namely, if $\ma U\in \R^{m\times n}$ spans an $n$-dimensional subspace, then
\begin{equation}
	g(\ma U^\trans \ve x) = c_0 + c_1 \ma U_{\cdot,1}^\trans \ve x + \ldots + c_n \ma U_{\cdot, n}^\trans \ve x 
				= c_0 + \left( \sum_{k=1}^n c_k \ma U_{\cdot,k} \right) \ve x
				= c_0 + \widehat{c}_1 \hma U^\trans \ve x,
\end{equation}
where $\hma U \in \R^{m\times 1}$ spans a one-dimensional subspace.

\section{Examples\label{sec:examples}}
To demonstrate the effectiveness of our algorithm for ridge approximation 
we apply it to a mixture of synthetic and application problems.
Code generating these examples along with an implementation of \cref{alg:main} are available at
{\tt \url{https://github.com/jeffrey-hokanson/varproridge}}.
First we compare the proposed algorithm to the alternating approach of~\cite{CEHW17}
on a problem where the ridge function is known a priori and examine both
convergence and wall clock time.
Next, we demonstrate that even though finding the ridge approximation is a non-convex problem,
the proposed algorithm frequently finds the global minimizer from a random initialization.
Then we study how rapidly the ridge approximation identifies the active subspace of a test problem
as the number of samples grows and compare these results to gradient based approaches for estimating the active subspace.
Finally, we apply the proposed algorithm to two application problems:
modeling lift and drag from a NACA0012 airfoil with 18 parameters
and modeling mean solution on the Neumann boundary of an elliptic PDE with 100-random coefficients.
We compare the accuracy of our ridge approximation to both Gaussian process and sparse surrogates
and find that polynomial ridge approximations provide more accurate approximations on
these application problems which exhibit ridge structure.

\subsection{Convergence of the optimization iteration}
As a first example, 
we compare the convergence of our proposed Gauss-Newton based algorithm 
to the alternating approach~\cite{CEHW17} on an examples
with and without a zero-residual solution.
The alternating approach switches between optimizing for the polynomial coefficients $\ve c$
and the subspace defined by $\ma U$ at each iteration; i.e.,
\begin{align}
	\ve c_k &\leftarrow \ma V(\ma U_{k-1})^+ \ve f \\
	\ma U_k &\leftarrow \minimize_{\Range \ma U \in \mathbb{G}(m,\R^n)}
			\frac12 \| \ve f - \ma V(\ma U)\ve c_k\|_2^2.
\end{align}
The implementation of Constantine et al.
allows the nonlinear optimization problem for $\ma U_k$ 
to be solved using multiple steps of steepest descent using {\tt pymanopt}~\cite{TKW16}.

\begin{figure}
\begin{tikzpicture}
	\begin{groupplot}[
		group style={group size = 2 by 2, group name=polynomial,
				vertical sep=10pt, horizontal sep=10pt,
			},
		xmin = 0, xmax = 40,
		ymode = log,
		ymax = 1, ymin = 1e-16,
		width = 0.52\linewidth, height = 0.35\linewidth, 
		title style = {yshift = -1ex}
		]
		\nextgroupplot[title = Gauss-Newton, ylabel = {$\| \ve r(\ma U_\ell) \|_2/\| \ve f\|_2$},
			ytickten = {-16, -14, ..., 2}, 
			xticklabels = {,,}]
		\pgfplotstableread{fig_polynomial_gauss_newton.dat}\gn
		\foreach \i in {0, 1, ..., 9}{
			\addplot[colorbrewerA1, mark = *, mark size = 1pt] table [x=iter, y expr = \thisrow{res\i}/ 192.70431446999882] {\gn};
		}
		
		\nextgroupplot[title = Alternating, ytickten = {-14, -12, ..., 2}, yticklabels = {,,}, xticklabels = {,,}]
		\pgfplotstableread{fig_polynomial_alternating.dat}\alternating
		\foreach \i in {0, 1, ..., 9}{
			\addplot[colorbrewerA2, mark = *, mark size = 1pt] table [x=iter, y expr = \thisrow{res\i}/ 192.70431446999882] {\alternating};
		}
		
		\nextgroupplot[ymin = 1e-2, ymax=1, xlabel = iteration, xlabel = iteration $\ell$, ylabel = {$\| \ve r(\ma U_\ell)\|_2/\|\ve f\|_2$}]
		\pgfplotstableread{fig_polynomial_noise_gauss_newton.dat}\gn
		\foreach \i in {0, 1, ..., 9}{
			\addplot[colorbrewerA1, mark = *, mark size = 1pt] table [x=iter, y expr = \thisrow{res\i}/192.482920096] {\gn};
		}
		\addplot[black, very thick] coordinates {(0, 3.03513499166/192.482920096) (50, 3.03513499166/192.482920096)};
		
		\nextgroupplot[ymin = 1e-2, ymax=1, yticklabels = {,,}, xlabel = iteration $\ell$ ]
		\pgfplotstableread{fig_polynomial_noise_alternating.dat}\alternating
		\foreach \i in {0, 1, ..., 9}{
			\addplot[colorbrewerA2, mark = *, mark size = 1pt] table [x=iter, y expr = \thisrow{res\i}/ 192.482920096] {\alternating};
		}
		\addplot[black, very thick] coordinates {(0, 3.03513499166/192.482920096) (50, 3.03513499166/192.482920096)};
			
	\end{groupplot}
	\node[anchor=south, rotate = 90, yshift = 40pt] at ($(polynomial c1r1.west)$){\footnotesize{zero residual}}; 
	\node[anchor=south, rotate = 90, yshift = 40pt] at ($(polynomial c1r2.west)$){\footnotesize{non-zero residual}}; 
\end{tikzpicture}
\caption{%
	A comparison of the per-iteration performance of our Gauss-Newton method (\cref{alg:main})
	and the alternating approach of Constantine, Eftekhari, Hokanson, and Ward~\cite[Alg.~2]{CEHW17}
	using $100$ steepest descent iterations per alternating iteration.
	The top rows compare these algorithms when finding a two variable cubic ridge function approximation
	from $M=1000$ uniform random samples of \cref{eq:convergence}.
	The bottom row shows the performance when independent and identically distributed Gaussian random noise was added to each sample of $f$
	to provide a non-zero residual solution;
	the black line indicates the norm of the noise introduced.
	}
\label{fig:convergence}
\end{figure}
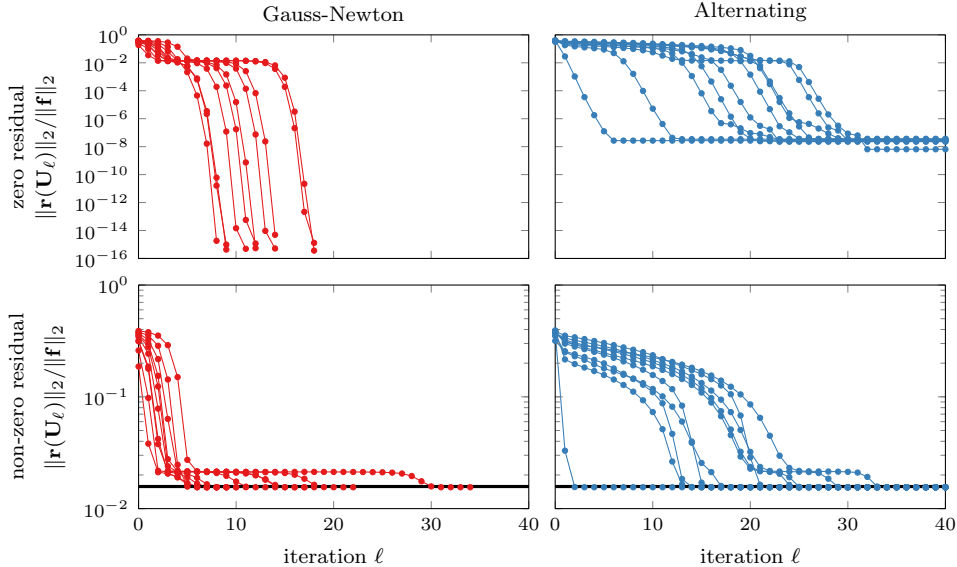

Existing results describe the convergence of both of these algorithms
for both the zero-residual and non-zero residual cases.
Iterates of a Gauss-Newton method
converging to a zero-residual solution do so quadratically
and those that converge to a non-zero residual solution do so only superlinearly~\cite[\S10.3]{NW06}.
Following Ruhe and Wedin~\cite[\S3]{RW80},
variable projection on a separable problem (in their notation, Algorithm I)
converges at the same rate as Gauss-Newton.
Hence, our ridge approximation algorithm should converge quadratically for zero residual problems
and superlinearly otherwise. % I think this due to the dropping of the second term in the Hessian
However an alternating approach (in their notation, Algorithm III),
such as the alternating approach for ridge approximation of Constantine et al.,
should converge at best only linearly.
Our numerical experiments support this analysis.
Figure~\ref{fig:convergence} compares the convergence of the alternating and Gauss-Newton approaches.
In this example, $f$ is taken to be a cubic ridge function on a two dimensional subspace:
\begin{equation}\label{eq:convergence}
	f: \set D \subset \R^{10} \to \R; \quad f(\ve x) = (\ve e_1^\trans \ve x)^2 + (\ve 1^\trans \ve x/10)^3 + 1; \quad
	\set D = [-1,1]^{10}.
\end{equation}
Sampling $M=1000$ points uniformly over the domain, 
Figure~\ref{fig:convergence} shows the per-iteration convergence history 
for ten different initializations of each algorithm.
As expected in the case with a zero-residual solution,
our proposed Gauss-Newton approach converges quadratically while the alternating approach converges only linearly.
The bottom row of this figure shows the case where Gaussian random noise with unit variance 
has been added to the function values $f(\ve x_i)$
ensuring there is not a zero-residual solution.
Although our method no longer converges quadratically, 
our method converges in fewer iterations on average than the alternating approach.

\Cref{fig:convergence} also exposes two interesting features of the optimization.
First, for both algorithms there is a plateau in the convergence history
at around $10^{-2}$ in the zero residual case and around $3\cdot 10^{-2}$ in the non-zero residual case.
For both algorithms, this happens when one of the two directions in the active subspace has been found.
Second, the alternating algorithm on the zero residual case stagnates with a residual around $10^{-8}$,
however the residual should be able to converge to an error of $10^{-14}$.
We suspect this is due to numerical issues in the implementation 
as the residual floor does not decrease as the termination criteria
are made more strict.

\subsection{Cost comparison}
\begin{table}
\footnotesize
\centering
\caption{Per iteration cost of Gauss-Newton and alternating based approaches,
including only $M$-dependent steps and treating evaluating the polynomial $\psi$ as an $\order(1)$ cost.
}
%Recall $M$ is the number of samples,
%$m$ the input dimension, 
%$n$ the ridge dimension, 
%$p$ the polynomial degree,
%and $N = {n + p \choose p}$ the dimension of the polynomial basis.
\label{tab:cost}
\begin{tabular}{llll}
	\toprule
	\multicolumn{2}{c}{Gauss-Newton} & \multicolumn{2}{c}{Alternating} \\
	\cmidrule(lr){1-2} 
	\cmidrule(lr){3-4} 
	step & cost & step & cost \\
	\midrule 
	fitting polynomial (SVD)& $\order(MN^2)$ &	
	fitting polynomial (QR)& $\order(MN^2)$ \\
	constructing Jacobian  & $\order(MNmn)$ &
	constructing Jacobian  & $\order(MNmn)$ \\
	Gauss-Newton direction & $\order(M(mn)^2)$ &
	steepest descent direction & $\order(MNmn)$ \\
	step acceptance & $\order(MN^2)$ &
	step acceptance & $\order(MN)$ \\
	\bottomrule
\end{tabular}

\end{table}

Comparing the cost of our variable projection Gauss-Newton based approach
to the alternating approach of Constantine et al.~\cite{CEHW17} is not simple.
Although both algorithms scale linearly with the number of samples $M$,
the constant multiplying $M$ depends on the dimension of the ridge approximation $n$, 
the polynomial degree $p$, the dimension of the polynomial basis $N={n+p \choose p}$,
and the dimension of the input space $m$.
\Cref{tab:cost} lists the dominant costs in each algorithm.
This analysis motivates using multiple steepest descent steps 
with a fixed polynomial per alternating iteration
as this step is cheap compared to the cost of the polynomial fitting step.
This cost analysis also suggests that for high dimensional input spaces,
the alternating approach may be cheaper as computing the descent direction
scales like $m^2$ in the Gauss-Newton approach vs $m$ in the steepest descent approach.
However, which algorithm is faster for a given set of $m$, $n$, and $p$
depends on the constants associated with each algorithm, which this analysis neglects.

\begin{table}
	\caption{Median wall clock time in seconds from ten replicates
		of each algorithm applied to identical zero-residual data consisting of $M=1000$ uniform samples of $f_{n,p}$ from~(\ref{eq:timing}),
		initialized with the same random subspace $\ma U$,
		and stopped when the normalized residual reached $10^{-5}$,
		which both algorithms achieved in \cref{fig:convergence}.
		The alternating method times represent the shortest time
		when using 1, 10, or 100 steps of steepest descent per alternating iteration.
		Experiments were conducted on a 2013 Mac Pro with a six-core Intel Xeon CPU E5-1650 v2 clocked at 3.50GHz
		with 16GB RAM.
	}
	\label{tab:timing}
	\footnotesize
	\centering
	\pgfplotstableset{fixed,fixed zerofill, precision=3,
		create on use/m1a/.style={create col/copy column from table={fig_timing_alt_best.dat}{m1}},
		create on use/m2a/.style={create col/copy column from table={fig_timing_alt_best.dat}{m2}},
		create on use/m3a/.style={create col/copy column from table={fig_timing_alt_best.dat}{m3}},
		create on use/m4a/.style={create col/copy column from table={fig_timing_alt_best.dat}{m4}},
		create on use/m5a/.style={create col/copy column from table={fig_timing_alt_best.dat}{m5}},
	}
	\begin{filecontents*}{fig_timing_gn.dat}
degree	m1	m2	m3	m4	m5	
2.0	0.0116817951202	0.0291941165924	0.0714950561523	0.107254981995	0.130165100098	
3.0	0.0129420757294	0.0683960914612	0.128842830658	0.492297887802	2.0233809948	
4.0	0.0146889686584	0.107215166092	0.186686992645	1.90059304237	2.98489904404	
5.0	0.01580119133	0.112193822861	0.363634824753	11.758808136	7.33096814156	
\end{filecontents*}
\pgfplotstabletypeset[
		columns/degree/.style={int detect, column name = {}},
		every head row/.style={
			before row={\toprule%
				$p$ & 
				\multicolumn{5}{c}{Gauss-Newton} &
				\multicolumn{5}{c}{Alternating}\\
				\cmidrule(lr){2-6}
				\cmidrule(lr){7-11}
			},
			after row={\midrule},
		},
		columns/m1/.style={column name={$n\!=\!1$}},
		columns/m2/.style={column name={$n\!=\!2$}},
		columns/m3/.style={column name={$n\!=\!3$}},
		columns/m4/.style={column name={$n\!=\!4$}},
		columns/m5/.style={column name={$n\!=\!5$}},
		columns/m1a/.style={column name={$n\!=\!1$}, precision = 1},
		columns/m2a/.style={column name={$n\!=\!2$}, precision = 1},
		columns/m3a/.style={column name={$n\!=\!3$}, precision = 1},
		columns/m4a/.style={column name={$n\!=\!4$}, precision = 1},
		columns/m5a/.style={column name={$n\!=\!5$}, precision = 1},
		columns = {degree, m1, m2, m3, m4, m5, m1a, m2a, m3a, m4a, m5a},
		every last row/.style={after row=\bottomrule},
	]{fig_timing_gn.dat}	
\end{table}

To get a sense of which algorithm is faster in practice, 
we compare the wall clock time of each algorithm on a $10$-dimensional test function
\begin{equation}\label{eq:timing}
	f_{n,p}: \set D\subset \R^{10} \to \R; \quad
	f_{n,p}(\ve x) = (\ve 1^\trans \ve x)^p + \sum_{j=1}^{n-1} (\ve e_j^\trans \ve x)^{p-1}; \quad
	\set D = [-1,1]^{10}
\end{equation}
fitting $M=1000$ points for a variety of polynomial degrees $p$
and subspace dimensions $n$
as show in \cref{tab:timing}.
As these results show, our proposed algorithm is significantly faster when measured in wall clock time.
Part of this improvement is due the quadratic convergence 
of our Gauss-Newton method that requires fewer iterations to terminate compared to the linearly convergent alternating method.
However, a nontrivial portion of performance difference is due to our more careful implementation
that is tightly coupled to the Grassmann optimization
and exploits the structure of the Jacobian revealed in \cref{thm:orth}.

\subsection{Convergence to the global minimizer}
As the polynomial ridge approximation problem~\cref{eq:pra_opt} is not necessarily a convex problem,
one concern is that our proposed algorithm might converge to a spurious local minimizer
rather than the global minimizer.
However, our numerical results suggest our algorithm frequently converges to the global minimizer regardless
of initial subspace estimate.
As an example, we consider a quadratic ridge approximation on an $n$-dimensional subspace
built from $M=1000$ samples of 
\begin{equation}\label{eq:initial}
	f_n: \set D\subset \R^{10} \to \R ; \quad
	f_n(\ve x) = \sum_{j=1}^n (\ve e_j^\trans \ve x)^2; \quad
	\set D = [-1,1]^{10}.
\end{equation}
As this function is a polynomial ridge function, 
we can assess if the algorithm has correctly converged if the residual is near zero.
Table~\ref{tab:initial} shows the frequency with which our proposed algorithm terminated with 
an incorrect subspace.
In general, this frequency is low and so it should be sufficient in many practical cases to
try multiple random initializations, taking the best to ensure convergence to an approximate
global minimizer.
\begin{table}
	\caption{Probability of not finding the global minimizer 
	given a random initial subspace of dimension $n$.
	This probability was estimated from $1000$ trials of Algorithm~\ref{alg:main}
	fitting a quadratic ridge function of $n$ variables to 
	$M=1000$ samples of $f(\ve x)$ from \cref{eq:initial} taken uniformly over the domain.
	}
	\label{tab:initial}
	\small
	\setlength{\tabcolsep}{5pt}
	\centering
	\begin{filecontents*}{fig_initial.dat}
d1	d2	d3	d4	d5	d6	d7	d8	d9	d10	
0.0	0.108	0.162	0.152	0.073	0.076	0.106	0.071	0.034	0.0	
\end{filecontents*}
\pgfplotstabletypeset[
		clear infinite,
		columns/d1/.style={column name = {$n=1$}},
		columns/d2/.style={column name = {$n=2$}},
		columns/d3/.style={column name = {$n=3$}},
		columns/d4/.style={column name = {$n=4$}},
		columns/d5/.style={column name = {$n=5$}},
		columns/d6/.style={column name = {$n=6$}},
		columns/d7/.style={column name = {$n=7$}},
		columns/d8/.style={column name = {$n=8$}},
		columns/d9/.style={column name = {$n=9$}},
		columns/d10/.style={column name = {$n=10$}},
		every column/.style = {
			column type=c,
			fixed, fixed zerofill, precision =1,
			dec sep align,
	        preproc/expr={100*##1},
	        postproc cell content/.append code={
    	    \ifnum1=\pgfplotstablepartno
        	   \pgfkeysalso{@cell content/.add={}{\%}}%
            \fi
	        },
		},
		columns = {d1, d2, d3, d4, d5, d6, d7, d8, d9, d10},
		every head row/.style={
			before row={\toprule},
			after row={\midrule},
		},
		every last row/.style={after row=\bottomrule},
	]{fig_initial.dat}	
\end{table}

\subsection{Convergence to the active subspace with increasing samples}
As one goal of building a ridge approximation is to construct an inexpensive surrogate
of an expensive function $f$, an important feature of this approximation 
is how many samples of this function are required to construct a good surrogate.
Here we compare the ridge subspace found by our polynomial ridge approximation
and the active subspace computed from the outer product of (approximate) gradients~\cite{Con15}
to the true active subspace of a toy problem
consisting of a one-dimensional quadratic ridge function plus 
low-amplitude sinusoidal oscillations:
\begin{equation}\label{eq:samples}
	f: [-1,1]^m \to \R, \qquad
	f(\ve x) = \frac12 (\ve 1^\trans \ve x)^2 + \alpha \sum_{j=1}^m \cos (\beta \pi [\ve x]_j),
	\qquad \alpha > 0, \quad \beta \in \mathbb{N}.
\end{equation}
These oscillations are necessary so that $f$ does not have an exact one-dimensional ridge structure,
in which case both the polynomial ridge approximation and the active subspace approach 
will correctly estimate the active subspace with only $M=m+3$ samples 
unless $\lbrace \ve x_i \rbrace_{i=1}^M$ is adversely chosen.
For this toy problem, the outer-product of gradients matrix has a closed form expression:
\begin{equation}
	\ma C := \int_\set D (\nabla f(\ve x)) (\nabla f(\ve x))^\trans \D \mu(\ve x)
		= \ve 1 \ve 1^\trans + (\alpha \beta \pi)^2 \ma I.
\end{equation}
The leading eigenvector of this matrix is always the ones vector for any value of $\alpha$ and $\beta$,
but $\alpha$ and $\beta$ affect each approach differently.
By increasing $\alpha$, the `noise'---the response that cannot be explained by a one-dimensional ridge approximation---increases.
By increasing $\beta$, the frequency of oscillations increase,
increasing the oscillations in the gradients appearing in $\ma C$.
Together, $\alpha$ and $\beta$ determine the first and second eigenvalues of $\ma C$,
namely $m + (\alpha\beta\pi)^2$ and $(\alpha\beta\pi)^2$,
and the relative eigenvalue gap determines the convergence of the Monte Carlo estimate of $\ma C$, $\hma C$~\cite[Cor.~3.10]{Con15}:
\begin{equation}
	\hma C :=\frac{1}{L} \sum_{i=1}^L (\nabla f(\ve x_i))(\nabla f(\ve x_i))^\trans.
\end{equation}

\begin{figure}
\centering 
\noindent
\begin{tikzpicture}
	
	\begin{groupplot}[
		group style = {group size = 2 by 2},
		height = 0.4\textwidth, 
		width = 0.5\textwidth, 
		xmin = 100, xmax = 1020,
		ymin = 0, ymax = 90,
		xtick = {100,200, 400, ..., 1000},
		ytick = {0,10,20,..., 90},
		domain = 100:1050,
		samples = 500,
		title style = {yshift=-7pt,},
		]
		\nextgroupplot[title = {$\alpha = 0.02$ $\beta = 1$}, ylabel = {subspace error, degrees},
			legend to name=samples]
		\addplot[colorbrewerA1, thick] table [x= M, y = p50] {fig_speed_pra_uniform_0.02_1.dat}
			[anchor = north east] node [pos = 1,yshift = 3pt] {\tiny uniform random};
		\addplot[colorbrewerA3, very thick] table [x= M, y = p50] {fig_speed_pra_cheat_0.02_1.dat}
			[anchor = south east] node [pos = 0.4] {\tiny ridge sample};
		\addplot[colorbrewerA2, thick] table [x= M, y = p50] {fig_speed_fd_grad_0.02_1.dat}
			[anchor = south east] node [pos = 1, yshift = 5pt] {\tiny active subspace};

		%\addlegendentry{ridge approximation, uniform random samples}
		%\addlegendentry{ridge approximation, uniform along ridge}
		%\addlegendentry{finite difference active subspace}
		% 1/sqrt(n) grid

		%\foreach \i in {100,200,..., 2000}{	
		%	\addplot[black, opacity = 0.2] {\i/sqrt(x)};
		%}

		\addplot[name path=prau1, draw=none] table [x=M,y=p25] {fig_speed_pra_uniform_0.02_1.dat}; 
		\addplot[name path=prau2, draw=none] table [x=M,y=p75] {fig_speed_pra_uniform_0.02_1.dat}; 
		\addplot[colorbrewerA1, opacity = 0.3] fill between [of = prau1 and prau2];
		
		\addplot[name path=fd1, draw=none] table [x=M,y=p25] {fig_speed_fd_grad_0.02_1.dat}; 
		\addplot[name path=fd2, draw=none] table [x=M,y=p75] {fig_speed_fd_grad_0.02_1.dat}; 
		\addplot[colorbrewerA2, opacity = 0.3] fill between [of = fd1 and fd2];
		
		\addplot[name path=cheat1, draw=none] table [x=M,y=p25] {fig_speed_pra_cheat_0.02_1.dat}; 
		\addplot[name path=cheat2, draw=none] table [x=M,y=p75] {fig_speed_pra_cheat_0.02_1.dat}; 
		\addplot[colorbrewerA3, opacity = 0.3] fill between [of = cheat1 and cheat2];

		\nextgroupplot[title = {$\alpha = 0.02$, $\beta = 5$}]
		\addplot[colorbrewerA1, thick] table [x= M, y = p50] {fig_speed_pra_uniform_0.02_5.dat}
			[anchor = south east] node [pos = 1,yshift = 3pt] {\tiny uniform random};
		\addplot[colorbrewerA3, very thick] table [x= M, y = p50] {fig_speed_pra_cheat_0.02_5.dat}
			[anchor = south east] node [pos = 0.4] {\tiny ridge sample};
		\addplot[colorbrewerA2, thick] table [x= M, y = p50] {fig_speed_fd_grad_0.02_5.dat}
			[anchor = south east] node [pos = 1, yshift = 15pt] {\tiny active subspace};

		\addplot[name path=prau1, draw=none] table [x=M,y=p25] {fig_speed_pra_uniform_0.02_5.dat}; 
		\addplot[name path=prau2, draw=none] table [x=M,y=p75] {fig_speed_pra_uniform_0.02_5.dat}; 
		\addplot[colorbrewerA1, opacity = 0.3] fill between [of = prau1 and prau2];
		
		\addplot[name path=fd1, draw=none] table [x=M,y=p25] {fig_speed_fd_grad_0.02_5.dat}; 
		\addplot[name path=fd2, draw=none] table [x=M,y=p75] {fig_speed_fd_grad_0.02_5.dat}; 
		\addplot[colorbrewerA2, opacity = 0.3] fill between [of = fd1 and fd2];
		
		%\addplot[name path=cheat1, draw=none] table [x=M,y=p25] {fig_speed_pra_cheat_0.004_1.dat}; 
		%\addplot[name path=cheat2, draw=none] table [x=M,y=p75] {fig_speed_pra_cheat_0.004_1.dat}; 
		%\addplot[colorbrewerA3, opacity = 0.3] fill between [of = cheat1 and cheat2];

		\nextgroupplot[title = {$\alpha = 0.004$, $\beta = 1$}, xlabel = {samples, $M$}, ylabel = {subspace error, degrees}]
		
		\addplot[colorbrewerA1, thick] table [x= M, y = p50] {fig_speed_pra_uniform_0.004_1.dat}
			[anchor = south west] node [pos = 0.4,yshift = 0pt] {\tiny uniform random};
		\addplot[colorbrewerA3, very thick] table [x= M, y = p50] {fig_speed_pra_cheat_0.004_1.dat}
			[anchor = south west] node [pos = 0.0, yshift = -5pt] {\tiny ridge sample};
		\addplot[colorbrewerA2, thick] table [x= M, y = p50] {fig_speed_fd_grad_0.004_1.dat}
			[anchor = south east] node [pos = 1, yshift = 0pt] {\tiny active subspace};

		\addplot[name path=prau1, draw=none] table [x=M,y=p25] {fig_speed_pra_uniform_0.004_1.dat}; 
		\addplot[name path=prau2, draw=none] table [x=M,y=p75] {fig_speed_pra_uniform_0.004_1.dat}; 
		\addplot[colorbrewerA1, opacity = 0.3] fill between [of = prau1 and prau2];
		
		\addplot[name path=fd1, draw=none] table [x=M,y=p25] {fig_speed_fd_grad_0.004_1.dat}; 
		\addplot[name path=fd2, draw=none] table [x=M,y=p75] {fig_speed_fd_grad_0.004_1.dat}; 
		\addplot[colorbrewerA2, opacity = 0.3] fill between [of = fd1 and fd2];
		
		\addplot[name path=cheat1, draw=none] table [x=M,y=p25] {fig_speed_pra_cheat_0.004_1.dat}; 
		\addplot[name path=cheat2, draw=none] table [x=M,y=p75] {fig_speed_pra_cheat_0.004_1.dat}; 
		\addplot[colorbrewerA3, opacity = 0.3] fill between [of = cheat1 and cheat2];

		\nextgroupplot[title = {$\alpha = 0.004$, $\beta = 5$}, xlabel = {samples, $M$}]

		\addplot[colorbrewerA1, thick] table [x= M, y = p50] {fig_speed_pra_uniform_0.004_5.dat}
			[anchor = south east] node [pos = 1,yshift = 0pt] {\tiny uniform random};
		\addplot[colorbrewerA3, very thick] table [x= M, y = p50] {fig_speed_pra_cheat_0.004_5.dat}
			[anchor = south west] node [pos = 0.0, yshift = -5pt] {\tiny ridge sample};
		\addplot[colorbrewerA2, thick] table [x= M, y = p50] {fig_speed_fd_grad_0.004_5.dat}
			[anchor = south east] node [pos = 1, yshift = 5pt] {\tiny active subspace};

		\addplot[name path=prau1, draw=none] table [x=M,y=p25] {fig_speed_pra_uniform_0.004_5.dat}; 
		\addplot[name path=prau2, draw=none] table [x=M,y=p75] {fig_speed_pra_uniform_0.004_5.dat}; 
		\addplot[colorbrewerA1, opacity = 0.3] fill between [of = prau1 and prau2];
		
		\addplot[name path=fd1, draw=none] table [x=M,y=p25] {fig_speed_fd_grad_0.004_5.dat}; 
		\addplot[name path=fd2, draw=none] table [x=M,y=p75] {fig_speed_fd_grad_0.004_5.dat}; 
		\addplot[colorbrewerA2, opacity = 0.3] fill between [of = fd1 and fd2];
		
		\addplot[name path=cheat1, draw=none] table [x=M,y=p25] {fig_speed_pra_cheat_0.004_5.dat}; 
		\addplot[name path=cheat2, draw=none] table [x=M,y=p75] {fig_speed_pra_cheat_0.004_5.dat}; 
		\addplot[colorbrewerA3, opacity = 0.3] fill between [of = cheat1 and cheat2];
	\end{groupplot}
	%\coordinate (c3) at ($(c1)!.5!(c2)$);
	%\node[below] at (c3 |- current bounding box.south) {\pgfplotslegendfromname{samples}};
\end{tikzpicture}
\caption{The subspace angle between the active subspace estimate $\ma U$ and its true value $\hma U = \ve 1/10$
	using $M$ samples with various methods for multiple random samples.
	In this example we use $f$ given in~\cref{eq:samples} on $\set D = [-1,1]^{100}$. 
	For each method, the median error is shown by a thick line and the 50\% interval by the shaded region.
	The active subspace estimate was computed using a one-sided finite difference from points randomly selected in the domain.
	The other two lines show the polynomial ridge approximation of degree $2$ constructed from different sampling schemes.
	The uniform random scheme chose points on $\set D$ with uniform weight
	whereas the ridge sample scheme chose points evenly spaced along the true active subspace
	and randomly with uniform probability in the orthogonal complement.
	} 
\label{fig:speed}
\end{figure}
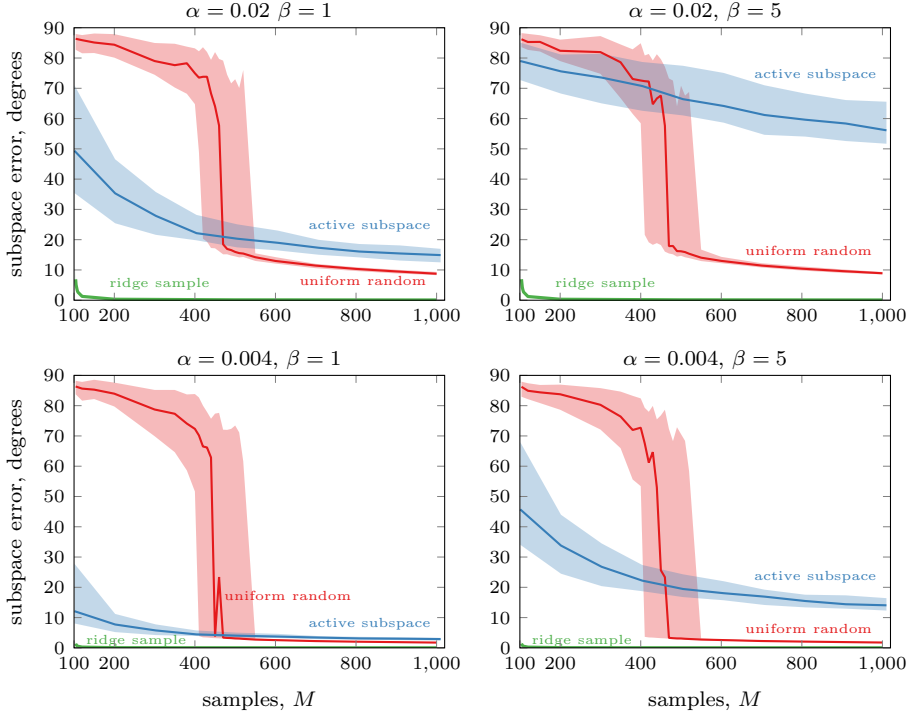

To provide a fair comparison in terms of function evaluations, we compare 
our polynomial ridge approximation approach to Monte Carlo estimate of the active subspace
using a finite difference approximation of the gradient:
\begin{equation}
	\tma C := \frac{1}{L}\sum_{i = 1}^{L} (\widetilde{\nabla} f(\ve x_i) ) (\widetilde{\nabla} f(\ve x_i))^\trans,
	\quad [\widetilde{\nabla} f(\ve x_i)]_j = \frac{ f(\ve x_i + h\ve e_j) - f(\ve x_i)}{h}.
\end{equation}
Although a finite difference gradient is used here, $f$ is smooth so the error introduced by the finite difference gradient is negligible.
As seen in \cref{fig:speed}, the Monte Carlo estimate converges at the expected $\order(M^{-1/2})$ rate.
When the relative gap remains the same, as in $\alpha = 0.02, \ \beta = 1$ and $\alpha = 0.004, \ \beta = 5$ cases of~\cref{eq:samples},
the convergence of the gradient based active subspace estimate is similar.

In contrast to the Monte Carlo estimate of the active subspace, 
the polynomial ridge approximation displays an interesting plateau, followed by apparent $\order(M^{-1/2})$ convergence.
During the plateau, the large angle with the true subspace is not (primarily) an artifact of a local minima;
even initializing with the true active subspace yields a subspace with large angles to the true active subspace.
This suggests that this plateau is likely due, loosely, to the information contained in the samples.
The later $\order(M^{-1/2})$ convergence we conjecture is due to the interpretation of the discrete least squares problem
as a Monte-Carlo approximation of the \emph{continuous} least squares problem:
\begin{equation}\label{eq:L2_err}
	\minimize_{\substack{g \in \mathbb{P}^p(\R^n) \\ \Range \ma U \in \mathbb{G}(n,\R^m)}}
		\int_{\set D} | f(\ve x) - g(\ma U^\trans \ve x)|^2 \D \mu(\ve x).
\end{equation}
The second ridge approximation built from samples uniformly along the true active subspace
suggests that a better sampling scheme the ridge approximation can convergence more rapidly.

\subsection{NACA airfoil\label{sec:examples:naca}}
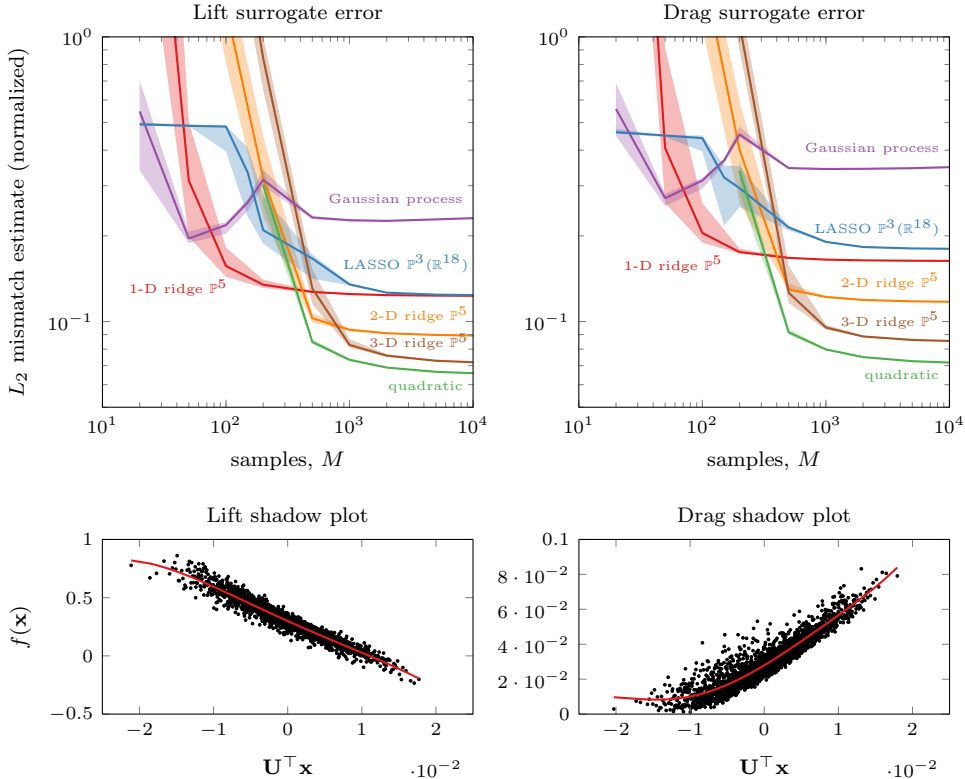
\begin{figure}
\noindent
\centering
\begin{tikzpicture}
	\begin{groupplot}[
		group style = {group size = 2 by 2, horizontal sep = 4em, vertical sep = 5em},
		width = 0.5\textwidth,
		height = 0.5\textwidth,
		xmode = log,
		ymode = log,
		ymax = 1,
		ymin = 5e-2,
		xmin = 10,
		xlabel = {samples, $M$},
		ylabel = {$L_2$ mismatch estimate (normalized)},
		title style = {yshift = -1ex},
		]
		\nextgroupplot[
			title = Lift surrogate error,
		]	
		\addplot[colorbrewerA1, thick] table [x = M, y = p50] {fig_naca_lift_pra_1_5.dat}
			[anchor = north east] node [pos = 0.5, yshift = 1pt] {\tiny 1-D ridge $\mathbb{P}^5$ };
		\addplot[name path=pra1, draw = none] table [x = M, y = p25] {fig_naca_lift_pra_1_5.dat};
		\addplot[name path=pra2, draw = none] table [x = M, y = p75] {fig_naca_lift_pra_1_5.dat};
		\addplot[colorbrewerA1, opacity = 0.3] fill between [of = pra1 and pra2];
		
		\addplot[colorbrewerA5, thick] table [x = M, y = p50] {fig_naca_lift_pra_2_5.dat}
			[anchor = south east] node [pos = 1, yshift = 1pt, xshift = 2pt] {\tiny 2-D ridge $\mathbb{P}^5$ };
		\addplot[name path=pra1, draw = none] table [x = M, y = p25] {fig_naca_lift_pra_2_5.dat};
		\addplot[name path=pra2, draw = none] table [x = M, y = p75] {fig_naca_lift_pra_2_5.dat};
		\addplot[colorbrewerA5, opacity = 0.3] fill between [of = pra1 and pra2];
		
		\addplot[colorbrewerA7, thick] table [x = M, y = p50] {fig_naca_lift_pra_3_5.dat}
			[anchor = south east] node [pos = 1, yshift = 1pt,xshift=2pt] {\tiny 3-D ridge $\mathbb{P}^5$ };
		\addplot[name path=pra1, draw = none] table [x = M, y = p25] {fig_naca_lift_pra_3_5.dat};
		\addplot[name path=pra2, draw = none] table [x = M, y = p75] {fig_naca_lift_pra_3_5.dat};
		\addplot[colorbrewerA7, opacity = 0.3] fill between [of = pra1 and pra2];
		
		\addplot[colorbrewerA4, thick] table [x = M, y = p50] {fig_naca_lift_gp.dat}
			[anchor = south east] node [pos = 1, yshift = 1pt] {\tiny Gaussian process };
		\addplot[name path=pra1, draw = none] table [x = M, y = p25] {fig_naca_lift_gp.dat};
		\addplot[name path=pra2, draw = none] table [x = M, y = p75] {fig_naca_lift_gp.dat};
		\addplot[colorbrewerA4, opacity = 0.3] fill between [of = pra1 and pra2];
		
		\addplot[colorbrewerA2, thick] table [x = M, y = p50] {fig_naca_lift_sp.dat}
			[anchor = south east] node [pos = 1, yshift = 5pt, xshift = 2pt] {\tiny LASSO $\mathbb{P}^3(\R^{18})$ };
		\addplot[name path=pra1, draw = none] table [x = M, y = p25] {fig_naca_lift_sp.dat};
		\addplot[name path=pra2, draw = none] table [x = M, y = p75] {fig_naca_lift_sp.dat};
		\addplot[colorbrewerA2, opacity = 0.3] fill between [of = pra1 and pra2];

		\addplot[colorbrewerA3, thick] table [x = M, y = p50] {fig_naca_lift_quad.dat}
			[anchor = north east] node [pos = 1, yshift = 1pt] {\tiny quadratic };
		\addplot[name path=pra1, draw = none] table [x = M, y = p25] {fig_naca_lift_quad.dat};
		\addplot[name path=pra2, draw = none] table [x = M, y = p75] {fig_naca_lift_quad.dat};
		\addplot[colorbrewerA3, opacity = 0.3] fill between [of = pra1 and pra2];

		%\addplot[colorbrewerA3, thick] table [x = M, y = p50] {fig_naca_lift_line.dat}
		%	[anchor = south east] node [pos = 1, yshift = 1pt] {\tiny linear };
		%\addplot[name path=pra1, draw = none] table [x = M, y = p25] {fig_naca_lift_line.dat};
		%\addplot[name path=pra2, draw = none] table [x = M, y = p75] {fig_naca_lift_line.dat};
		%\addplot[colorbrewerA3, opacity = 0.3] fill between [of = pra1 and pra2];

		\nextgroupplot[title = Drag surrogate error, ylabel = {},]

		\addplot[colorbrewerA1, thick] table [x = M, y = p50] {fig_naca_drag_pra_1_5.dat}
			[anchor = north east] node [pos = 0.6, yshift = 1pt] {\tiny 1-D ridge $\mathbb{P}^5$ };
		\addplot[name path=pra1, draw = none] table [x = M, y = p25] {fig_naca_drag_pra_1_5.dat};
		\addplot[name path=pra2, draw = none] table [x = M, y = p75] {fig_naca_drag_pra_1_5.dat};
		\addplot[colorbrewerA1, opacity = 0.3] fill between [of = pra1 and pra2];
		
		\addplot[colorbrewerA5, thick] table [x = M, y = p50] {fig_naca_drag_pra_2_5.dat}
			[anchor = south east] node [pos = 1, yshift = 1pt] {\tiny 2-D ridge $\mathbb{P}^5$ };
		\addplot[name path=pra1, draw = none] table [x = M, y = p25] {fig_naca_drag_pra_2_5.dat};
		\addplot[name path=pra2, draw = none] table [x = M, y = p75] {fig_naca_drag_pra_2_5.dat};
		\addplot[colorbrewerA5, opacity = 0.3] fill between [of = pra1 and pra2];
		
		\addplot[colorbrewerA7, thick] table [x = M, y = p50] {fig_naca_drag_pra_3_5.dat}
			[anchor = south east] node [pos = 1, yshift = 1pt] {\tiny 3-D ridge $\mathbb{P}^5$ };
		\addplot[name path=pra1, draw = none] table [x = M, y = p25] {fig_naca_drag_pra_3_5.dat};
		\addplot[name path=pra2, draw = none] table [x = M, y = p75] {fig_naca_drag_pra_3_5.dat};
		\addplot[colorbrewerA7, opacity = 0.3] fill between [of = pra1 and pra2];
		
		\addplot[colorbrewerA4, thick] table [x = M, y = p50] {fig_naca_drag_gp.dat}
			[anchor = south east] node [pos = 1, yshift = 1pt] {\tiny Gaussian process };
		\addplot[name path=pra1, draw = none] table [x = M, y = p25] {fig_naca_drag_gp.dat};
		\addplot[name path=pra2, draw = none] table [x = M, y = p75] {fig_naca_drag_gp.dat};
		\addplot[colorbrewerA4, opacity = 0.3] fill between [of = pra1 and pra2];
		
		\addplot[colorbrewerA2, thick] table [x = M, y = p50] {fig_naca_drag_sp.dat}
			[anchor = south east] node [pos = 1, yshift = 1pt] {\tiny LASSO $\mathbb{P}^3(\R^{18})$};
		\addplot[name path=pra1, draw = none] table [x = M, y = p25] {fig_naca_drag_sp.dat};
		\addplot[name path=pra2, draw = none] table [x = M, y = p75] {fig_naca_drag_sp.dat};
		\addplot[colorbrewerA2, opacity = 0.3] fill between [of = pra1 and pra2];

		\addplot[colorbrewerA3, thick] table [x = M, y = p50] {fig_naca_drag_quad.dat}
			[anchor = north east] node [pos = 1, yshift = 1pt] {\tiny quadratic };
		\addplot[name path=pra1, draw = none] table [x = M, y = p25] {fig_naca_drag_quad.dat};
		\addplot[name path=pra2, draw = none] table [x = M, y = p75] {fig_naca_drag_quad.dat};
		\addplot[colorbrewerA3, opacity = 0.3] fill between [of = pra1 and pra2];

		%\addplot[colorbrewerA3, thick] table [x = M, y = p50] {fig_naca_drag_line.dat}
		%	[anchor = south east] node [pos = 1, yshift = 1pt] {\tiny linear };
		%\addplot[name path=pra1, draw = none] table [x = M, y = p25] {fig_naca_drag_line.dat};
		%\addplot[name path=pra2, draw = none] table [x = M, y = p75] {fig_naca_drag_line.dat};
		%\addplot[colorbrewerA3, opacity = 0.3] fill between [of = pra1 and pra2];
		\nextgroupplot[
			height = 0.3\textwidth,
			title = Lift shadow plot,
			xlabel = {$\ma U^\trans \ve x$},
			ylabel = $f(\ve x)$,
			xmode = linear,
			ymode = linear,
			xmin = -0.025,
			xmax = 0.025,
			ymin = -0.5,
			ymax = 1,
			clip mode=individual,
		]
		\addplot[black, only marks, mark size = 0.5pt] table [x=UX, y = fX] {fig_naca_lift_ridge.dat};
		\addplot[colorbrewerA1, thick] table [x=UX, y = y] {fig_naca_lift_ridge.dat};

		\nextgroupplot[
			height = 0.3\textwidth,
			title = Drag shadow plot,
			xlabel = {$\ma U^\trans \ve x$},
			ylabel = {},
			%ylabel = $f(\ve x)$,
			xmode = linear,
			ymode = linear,
			xmin = -0.025,
			xmax = 0.025,
			ymin = 0,
			ymax = 0.1,
			clip mode=individual,
		]
		\addplot[black, only marks, mark size = 0.5pt] table [x=UX, y = fX] {fig_naca_drag_ridge.dat};
		\addplot[colorbrewerA1, thick] table [x=UX, y = y] {fig_naca_drag_ridge.dat};
	\end{groupplot}
\end{tikzpicture}

\caption{The top two plots show the estimated $L_2$ error using Monte Carlo integration for several different surrogate models
applied to the lift and drag of a NACA0012 airfoil as described in \cref{sec:examples:naca}.
The solid line indicates the median mismatch and the shaded region encloses the $25$th to $75$th percentile
of mismatch from 100 fits using randomly selected samples.
The bottom two plots show the shadow of these high dimensional points onto the 1-D ridge subspace fit with a 5th degree polynomial.
}
\label{fig:naca}
\end{figure}
As a first demonstration of our algorithm on an application problem,
we consider an 18-parameter model of a NACA0012 airfoil with two quantities of interest:
the nondimensionalized lift and drag coefficients.
This model from~\cite[\S5.3.1]{Con15} depends on 18 Hicks-Henne parameters that modify the airfoil geometry
and both lift and drag are computed using the Stanford University Unstructured (SU2) computational fluid dynamics code~\cite{economon2016}.
\Cref{fig:naca} shows the estimated $L_2$ mismatch 
normalized by the $L_2$ norm of $f$; 
both integrals are estimated using Monte Carlo, cf.~\cref{eq:L2_err}.
In addition to polynomial ridge approximations of various subspace dimensions,
this example compares two other surrogate models:
a Gaussian process model using {\tt sklearn}'s {\tt GaussanProcessRegressor}
and a global cubic model with a sparsity encouraging $\ell_1$ penalty using {\tt LassoCV}
which includes cross-validation to pick the regularization parameter~\cite{sklearn}.
As these results show, a 1-D polynomial ridge approximation does well with limited samples,
providing a better surrogate than either a Gaussian process or a sparse approximation.
However, increasing the subspace dimension does not significantly improve the fit
and quadratic polynomial using all the input coordinates provides the best surrogate with a large number of samples.

\subsection{Elliptic PDE\label{sec:examples:pde}}
\begin{figure}
\noindent
\setlength\tabcolsep{2pt}
\begin{tabular}[t]{ll}
\begin{tikzpicture}
	\begin{groupplot}[
		group style = {group size = 1 by 2, vertical sep = 3.5em},
		title style = {yshift = -2ex},
		xlabel style = {yshift = 1ex},
	]
	\nextgroupplot[	
		width=0.47\textwidth,
		height =0.35\textwidth,
		title = Subspace shadow plot,
		xlabel = {$\ma U^\trans \ve x$},
		ylabel = $f(\ve x)$,
		ymin = 0,
		clip mode=individual,
	]
	\addplot[black, only marks, mark size = 0.5pt] table [x=UX, y = fX] {fig_pde_ridge.dat};
	\addplot[colorbrewerA1, thick] table [x=UX, y = y] {fig_pde_ridge.dat};

	\nextgroupplot[
		width = 0.47\textwidth,
		height = 0.3\textwidth,
		ymin = -1, ymax = 1,
		title = Ridge subspace,
		xlabel = $i$,
		ylabel = $U_i$,
		xmax = 100,
		ylabel style = {yshift = -2ex},
		clip mode=individual,
		]
	\addplot[black, only marks, mark size = 1pt] table [x = i, y = Ui] {fig_pde_ridge_U.dat};
	\end{groupplot}
\end{tikzpicture}
&
\begin{tikzpicture}
\begin{axis}[ymode = log, xtick = {1e2,1e3,1e4},
	xmin = 100,
	xmax = 50000,
	ymin = 1e-2, 
	ymax = 2,
	xlabel = {samples, $M$},
	ylabel = {Monte Carlo $L_2$ mismatch estimate (normalized)},
	height = 0.61\textwidth,
	width = 0.5\textwidth,
	xmode = log,
	ylabel style = {yshift = -1ex},
	title = {Surrogate error},
	title style = {yshift = -2ex},
	xtick = {1e2, 1e3, 1e4, 5e4},
	xticklabels = {$10^2$, $10^3$, $10^4$, $5\cdot 10^4$},
	xlabel style = {yshift = 1ex},
	]
	\addplot[colorbrewerA5, thick] table [x = M, y = p50] {fig_pde_long_pra_1_3.dat}
		[anchor = south east] node [pos = 1, yshift = 1pt] {\tiny 1-D ridge $\mathbb{P}^3$ };
	\addplot[name path=pra1, draw = none] table [x = M, y = p25] {fig_pde_long_pra_1_3.dat};
	\addplot[name path=pra2, draw = none] table [x = M, y = p75] {fig_pde_long_pra_1_3.dat};
	\addplot[colorbrewerA5, opacity = 0.3] fill between [of = pra1 and pra2];

	\addplot[colorbrewerA1, thick] table [x = M, y = p50] {fig_pde_long_pra_1_7.dat}
		[anchor = north east] node [pos = 0.8, yshift = 0pt] {\tiny  1-D ridge $\mathbb{P}^7$ };
	\addplot[name path=pra1, draw = none] table [x = M, y = p25] {fig_pde_long_pra_1_7.dat};
	\addplot[name path=pra2, draw = none] table [x = M, y = p75] {fig_pde_long_pra_1_7.dat};
	\addplot[colorbrewerA1, opacity = 0.3] fill between [of = pra1 and pra2];

	\addplot[colorbrewerA2, thick] table [x = M, y = p50] {fig_pde_long_sp.dat}
		[anchor = south east] node [pos = 1, yshift = 1pt,xshift = 20pt] {\tiny LASSO $\mathbb{P}^3(\R^{100})$ };
	\addplot[name path=sp1, draw = none] table [x = M, y = p25] {fig_pde_long_sp.dat};
	\addplot[name path=sp2, draw = none] table [x = M, y = p75] {fig_pde_long_sp.dat};
	\addplot[colorbrewerA2, opacity = 0.3] fill between [of = sp1 and sp2];

	\addplot[colorbrewerA3, thick] table [x = M, y = p50] {fig_pde_long_line.dat}
		[anchor = south east] node [pos = 1, yshift = 1pt] {\tiny linear };
	\addplot[name path=line1, draw = none] table [x = M, y = p25] {fig_pde_long_line.dat};
	\addplot[name path=line2, draw = none] table [x = M, y = p75] {fig_pde_long_line.dat};
	\addplot[colorbrewerA3, opacity = 0.3] fill between [of = line1 and line2];

	\addplot[colorbrewerA4, thick] table [x = M, y = p50] {fig_pde_long_gp.dat}
		[anchor = south east] node [pos = 1, yshift = 0pt] {\tiny Gaussian process };
	\addplot[name path=gp1, draw = none] table [x = M, y = p25] {fig_pde_long_gp.dat};
	\addplot[name path=gp2, draw = none] table [x = M, y = p75] {fig_pde_long_gp.dat};
	\addplot[colorbrewerA4, opacity = 0.3] fill between [of = gp1 and gp2];

	\addplot[colorbrewerA7, thick] table [x = M, y = p50] {fig_pde_long_pra_2_7.dat}
		[anchor = south east] node [pos = 1, yshift = 25pt, xshift = 0pt] {\tiny 2-D ridge $\mathbb{P}^7$};
	\addplot[name path=pra1, draw = none] table [x = M, y = p25] {fig_pde_long_pra_2_7.dat};
	\addplot[name path=pra2, draw = none] table [x = M, y = p75] {fig_pde_long_pra_2_7.dat};
	\addplot[colorbrewerA7, opacity = 0.3] fill between [of = pra1 and pra2];

\end{axis}
\end{tikzpicture}
\end{tabular}

\caption{The top left plot shows the 1-D ridge function approximation (solid line)
along with the projected points $y_i = \ma U^\trans \ve x_i \in \R$ (dots).
The bottom left plot shows the entries in the matrix $\ma U \in \R^{100\times 1}$.
The right plot shows the estimated $L_2$ error using Monte Carlo integration for several different surrogate models
applied to the elliptic PDE model described in \cref{sec:examples:pde}
where the solid line indicates the median mismatch and the shaded region encloses the $25$th to $75$th percentile
mismatch from 100 fits with different random samples. 
}
\label{fig:pde}
\end{figure}
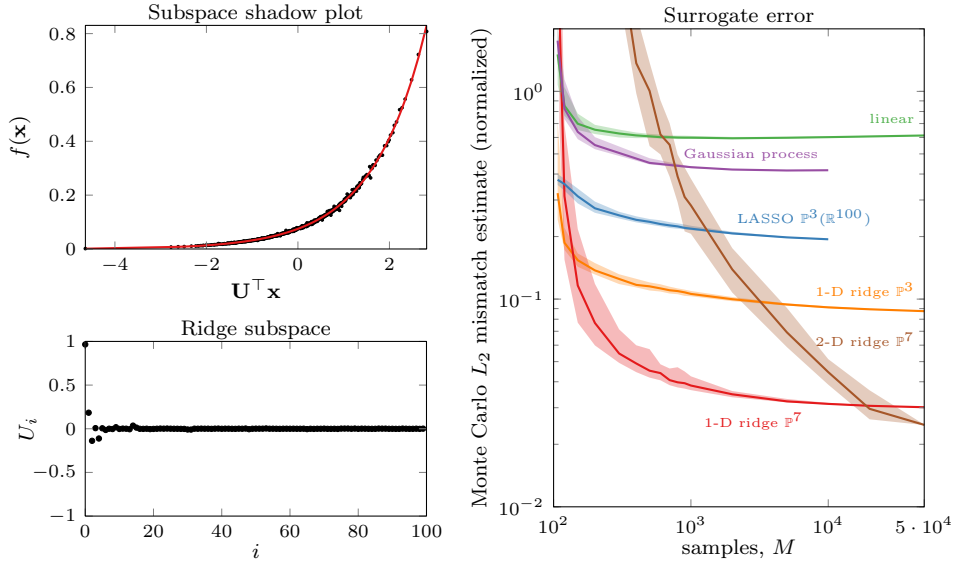
As a final demonstration of our algorithm, we consider a 2D elliptic PDE from~\cite{CDW14} with $m=100$ input parameters.
These parameters $\ve x\in \R^{100}$ characterize the coefficient $a = a(\ve s, \ve x)$
in the differential equation
\begin{equation}\label{eq:pde}
	-\nabla_\ve s\cdot(a \nabla_\ve s u) = 1,\quad \ve s \in [0,1]^2,
\end{equation}
where the inputs $x_j$ are the coefficients of a truncated Karhunen-Lo\'{e}ve expansion of $\log(a)$ with a correlation function,
\begin{equation}
\operatorname{Corr}(\ve s_1,\ve s_2) \;=\;
\exp\left(
\frac{-\| \ve s_1 - \ve s_2\|_1}{2\ell}
\right),
\end{equation}
where $\ell$ is a correlation length parameter. 
The boundary conditions are homogeneous Neumann conditions on the right boundary and 
homogeneous Dirichlet conditions on the other boundaries;
the quantity of interest is the spatial average of the solution $u$ on the right boundary. 
\Cref{fig:pde} shows the shadow plot for this quantity of interest,
which exhibits a strong, 1-D ridge structure where the coordinates of the active subspace spanned by $\ma U$
are approximately sparse.
As this figure illustrates, by exploiting the approximate ridge structure present in this quantity of interest,
we can form an accurate surrogate using relatively few samples.

\section{Summary and discussion}
Here we have derived a structure exploiting algorithm to efficiently solve
the data-driven polynomial ridge approximation problem~\cref{eq:pra_opt}.
The key feature of this algorithm is exploiting the separable structure
which allows us to optimize over the subspace alone by 
implicitly solving for the polynomial approximation using variable projection.
This allows our Gauss-Newton based optimization over the Grassmann manifold to display 
superior convergence properties compared to the alternating method of Constantine et al~\cite{CEHW17}
and allows us to reduce computational costs by exploiting the orthogonality properties of the Jacobian
revealed in \cref{thm:orth}.

This combination of variable projection and manifold optimization can likely used to accelerate optimization
of other surrogate model classes.
For example, we could replace the polynomial model of total degree $p$ with a tensor product spline model
which also yields a linear least squares problem to recover $g$.
However, as this spline model is not rotationally invariant, 
the optimization would have to be with respect to the Stiefel manifold.
Or, we could build $g$ using a Gaussian process model on the projected points $\lbrace\ve y_i = \ma U^\trans \ve x_i\rbrace_{i=1}^M$
and apply a regularization technique to ensure we do not obtain an interpolant for any choice of $\ma U$.
However in this case we could no longer use variable projection as the analog of $\ma V(\ma U)$ would be a square, full rank matrix.
Similar techniques could also be extended to the projection pursuit model~\cref{eq:ppr},
but then optimization would be over the product of $n$ one-dimensional Grassmann manifolds.

Beyond the scope of this work are two more fundamental questions about constructing surrogates,
and in particular, ridge approximations:
(i) how do we select the `best' choice of polynomial degree $p$ and subspace dimension $n$
and (ii) how do we choose our samples $\lbrace \ve x_i\rbrace_{i=1}^M\subset \set D$ 
to maximize the accuracy of our ridge approximation?
When $f(\ve x)$ contains random noise with a known distribution,
as is the case in statistical regression,
there are existing approaches to answer both these questions.
The hyperparameters $n$ and $p$ can be chosen using a number of techniques 
such as the Akaike information criterion (AIC)~\cite{Aka74}
or cross-validation~\cite{PC84}.
With this assumption of noise, we can also invoke traditional experimental design techniques
to choose points $\lbrace \ve x_i \rbrace_{i=1}^M$ that 
minimize the variance of our parameter estimates $\ma U$ and $g$~\cite[Ch.~6]{SWN03}.
However, the quantities of interest that often appear in uncertainty quantification 
do not have statistical noise but instead often display structured artifacts due to mesh discretizations and solver tolerances~\cite{MW11}.
But, as \cref{fig:speed} suggests, if a good sequential point selection heuristic can be determined,
this would enable the construction of a better ridge approximation with fewer function evaluations.
This remains an active area of research.

\section*{Acknowledgements}
The authors would like to thank Akil Narayan
for his suggestion to consider orthogonal polynomial bases,
such as the Legendre basis.

\bibliographystyle{siamplain}
\bibliography{abbrevjournals,master}
\end{document}